\documentclass[11pt]{amsart}
\usepackage[margin=1.15in]{geometry}
\usepackage{amsmath, amssymb, amsthm, amsfonts,  wasysym}
\usepackage{mathrsfs}  
\usepackage{graphicx}
\usepackage{tabularx}
\usepackage[utf8]{inputenc}
\usepackage[all]{xy}  
\usepackage[dvipsnames]{xcolor}
\usepackage{tikz-cd}
\usepackage{eucal}  
\usepackage{latexsym}
\usepackage[new]{old-arrows}

\usepackage{verbatim}

\usepackage{bookmark}

\usepackage{hyperref}
\hypersetup{
     colorlinks=true,
     linkcolor=NavyBlue,
     filecolor=NavyBlue,
     citecolor = TealBlue,
     urlcolor=magenta,
 }

\usepackage{enumerate}

\newcommand{\kk}{\mathbf{k}}   
\newcommand{\nP}{\mathbf{P}}

\newcommand{\sO}{\mathscr{O}}                    % structure sheaf
\newcommand{\sI}{\mathscr{I}} 
\newcommand{\sF}{\mathscr{F}}   

\newcommand\pr{\operatorname{pr}}
\newcommand\id{\operatorname{id}}

\newcommand\Sing{\operatorname{Sing}}
\newcommand\codim{\operatorname{codim}}
\newcommand\rank{\operatorname{rank}}
\newcommand\corank{\operatorname{corank}}

\newcommand\gon{\operatorname{gon}}
\newcommand\A[1]{{\textstyle\bigwedge^{#1}}}

\newcommand\Supp{\textup{Supp\,}}

\theoremstyle{plain}
\newtheorem{Thm}{Theorem}[section]
\newtheorem{Prop}[Thm]{Proposition}

\newtheorem{Lem}[Thm]{Lemma}

\theoremstyle{definition}
\newtheorem{Def}[Thm]{Definition}
\newtheorem{Q}[Thm]{Question}
\newtheorem{problem}[Thm]{Problem}

\theoremstyle{remark}
\newtheorem{Rmk}[Thm]{Remark}
\newtheorem{Ex}[Thm]{Example}

\numberwithin{equation}{section}

\begin{document}

\title[Syzygies of secant varieties of curves and gonality sequences]{Syzygies of secant varieties of smooth projective curves and gonality sequences}

\author{Junho Choe}
\address{School of Mathematics, KIAS, 85 Hoegiro Dongdaemun-gu, Seoul 02455, Republic of Korea.}
\email{junhochoe@kias.re.kr}

\author{Sijong Kwak}
\address{Department of Mathematical Sciences, KAIST, 291 Daehak-ro, Yuseong-gu, Daejeon 34141, Republic of Korea}
\email{sjkwak@kaist.ac.kr}

\author{Jinhyung Park}
\address{Department of Mathematical Sciences, KAIST, 291 Daehak-ro, Yuseong-gu, Daejeon 34141, Republic of Korea}
\email{parkjh13@kaist.ac.kr}

\begin{abstract}
The purpose of this paper is to prove that one can read off the gonality sequence of a smooth projective curve from syzygies of secant varieties of the curve embedded by a line bundle of sufficiently large degree. More precisely, together with Ein--Niu--Park's theorem, our main result  shows that the gonality sequence of a smooth projective curve completely determines the shape of the minimal free resolutions of secant varieties of the curve of sufficiently large degree. This is a natural generalization of the gonality conjecture on syzygies of smooth projective curves established by Ein--Lazarsfeld and Rathmann to the secant varieties. 
\end{abstract}

\date{\today}
\thanks{J. Choe was supported by a KIAS Individual Grant (MG083301) at Korea Institute for Advanced Study. J. Park was partially supported by the National Research Foundation (NRF) funded by the Korea government (MSIT) (NRF-2021R1C1C1005479 and NRF-2022M3C1C8094326).}

\maketitle

%\tableofcontents 

%\setcounter{page}{1}

%%%%%%%%%%%%%%%%%%%%%%%%%%%%%%%%%%%%%%%%%%%%%%%%%%%%%%%%
\section{Introduction}
%%%%%%%%%%%%%%%%%%%%%%%%%%%%%%%%%%%%%%%%%%%%%%%%%%%%%%%%
Exploring the interplay between the geometric properties of algebraic varieties and the algebraic properties of the equations defining algebraic varieties is an important subject in algebraic geometry. Along this line, Green \cite{Green} and Ein--Lazarsfeld \cite{EL1} suggested studying syzygies of algebraic varieties from a geometric point of view. The case of smooth projective curves of large degree is now fairly well understood (see \cite{EL2}, \cite{Green}, \cite{GL}, \cite{Rathmann}). On the other hand, there has been a great deal of work on secant varieties of projective varieties in the last three decades particularly because some results on secant varieties have found some nontrivial applications to algebraic statistics and algebraic complexity theory. 
In \cite{CK}, Choe--Kwak observed that  there should be a ``matryoshka structure'' among secant varieties of projective varieties. In this paper, we show that the gonality sequence of a smooth projective curve gives a ``matryoshka structure'' for secant varieties of the curve. Our theorem is a natural generalization of the results of Ein--Lazarsfeld \cite{EL3} and Rathmann \cite{Rathmann} on the gonality conjecture of curves, and complements the theorem of Ein--Niu--Park \cite{ENP1} on syzygies of secant varieties of curves.

\medskip

Throughout the paper, we work over an algebraically closed field $\kk$ of characteristic zero. We start by recalling basic notions of syzygies of algebraic varieties. Let $X$ be a projective variety, and $L$ be a very ample line bundle on $X$ giving an embedding
$$
X \subseteq \nP H^0(X, L) = \nP^r.
$$
For a coherent sheaf $B$ on $X$, when $H^0(X,B\otimes L^m)=0$ for all $m\ll 0$, the section module $R=R(X, B; L):=\bigoplus_{m \in\mathbf{Z}}H^0(X,B \otimes L^m)$ is a finitely generated graded module over the homogeneous coordinate ring $S:=\bigoplus_{m \in\mathbf{Z}} S^m H^0(X,L)$ of $\nP^r$. By Hilbert's syzygy theorem, $R$ is minimally resolved as
$$
\begin{tikzcd}
0 & R \ar[l] & F_0 \ar[l] & F_1 \ar[l] & ~\cdots~ \ar[l] & F_r \ar[l] & 0 \ar[l],
\end{tikzcd}
$$
where
$$
F_p=\bigoplus_{q\in\mathbf{Z}}K_{p,q}(X,B; L)\otimes S(-p-q).
$$
Here $K_{p,q}(X, B; L)$ can be regarded as the space of $p$-th syzygies of weight $q$ of $R$ over $S$. This is called the \emph{Koszul cohomology}, and it is the cohomology of the following Koszul-type complex
\begin{multline*}
\A{p+1} H^0(X, L) \otimes H^0(X, B \otimes L^{q-1}) \longrightarrow  \A{p} H^0(X, L) \otimes H^0(X, B \otimes L^q)\\
\longrightarrow \A{p-1} H^0(X, L) \otimes H^0(X, B \otimes L^{q+1}).
\end{multline*}
For simplicity, we set $R(X, L):=R(X, \sO_X; L)$ and $K_{p,q}(X, L):=K_{p,q}(X, \sO_X; L)$.
For a globally generated vector bundle $E$, let $M_E$ be the kernel bundle of the evaluation map $\operatorname{ev} \colon H^0(X, E) \otimes \sO_X \to E$. 
It is well known (cf. \cite[Section 2.1]{AN}, \cite[Proposition 2.1]{Park}) that if $H^i(X, B \otimes L^m)=0$ for $i>0, m>0$, then 
$$
K_{p,q}(X, B; L) = H^1(X, \A{p+1}M_L \otimes B \otimes L^{q-1}) = \cdots = H^{q-1}(X, \A{p+q-1} M_L \otimes B \otimes L)
$$
for $p \geq 0$ and $q \geq 2$.
If furthermore $H^i(X, B)=0$ for $\max\{1,q-1\}  \leq i \leq q$ when $q \geq 1$ or $H^0(X, B \otimes L^{-1})=0$ when $q=0$, then
$$
K_{p,q}(X, B; L) = H^q(X, \A{p+q} M_L \otimes B)
$$
for $p \geq 0$ and $q \geq 0$.
The geometric approach to vanishing or nonvanishing of $K_{p,q}(X,B; L)$ is an active research subject in algebraic geometry.

\medskip

We now turn to the main objects of the paper -- the secant varieties of curves. Let $C$ be a smooth projective curve of genus $g$, and $L$ be a very ample line bundle on $C$ giving an embedding
$$
C \subseteq \nP H^0(C, L) = \nP^r.
$$
For an integer $k \geq 0$, assume that
$$
\deg L \geq 2g+2k+1,
$$
and consider the \emph{$k$-th secant variety}
$$
\Sigma_k:= \overline{\bigcup_{x_i \in C} \langle x_1,\ldots, x_{k+1} \rangle}\subseteq \nP^r
$$
of $C$ in $\nP^r$. In our case, $\Sigma_k$ is simply the union of $(k+1)$-secant $k$-planes to $C$ in $\nP^r$. 
We have the natural inclusions
$$
C=\Sigma_0 \subseteq \Sigma_1 \subseteq \cdots \subseteq \Sigma_{k-1} \subseteq \Sigma_k.
$$ 
Note that $\dim \Sigma_k = 2k+1$ and $\Sing \Sigma_k = \Sigma_{k-1}$. By \cite[Theorem 1.1]{ENP1}, $\Sigma_k$ has normal Du Bois singularities. The singularities of the first secant varieties of smooth projective varieties were before studied in \cite{CS} and \cite{Ullery}.

\medskip

A celebrated theorem of Green \cite{Green} asserts that if $\deg L \geq 2g+1+\ell$ for $\ell \geq 0$, then $C \subseteq \nP^r$ is arithmetically Cohen--Macaulay and $L$ satisfies property $N_{2,\ell}$; in other words,
$$
K_{p,q}(C, L) = 0~~\text{ for $0 \leq p \leq \ell$ and $q \geq 2$}.
$$
A natural generalization of Green's theorem to secant varieties of $C$ in $\nP^r$ was conjectured by Sidman--Vermeire \cite[Conjecture 1.3]{SV1}, and this was established by Ein, Niu, and the third author \cite[Theorem 1.2]{ENP1}: If $\deg L \geq 2g+2k+1+\ell$ for $\ell \geq 0$, then $\Sigma_k \subseteq \nP^r$ is arithmetically Cohen--Macaulay, and $\sO_{\Sigma_k}(1)$ satisfies property $N_{k+2, \ell}$; in other words,
$$
K_{p,q}(\Sigma_k, \sO_{\Sigma_k}(1)) = 0~~\text{ for $0 \leq p \leq \ell$ and $q \geq k+2$}.
$$
By specializing Danila's theorem \cite{Danila} to the curve case (see \cite{ENP2}), we have
$$
H^0(\Sigma_k, \sO_{\Sigma_k}(m))=H^0(\nP^r, \sO_{\nP^r}(m))~~\text{ for $0 \leq m \leq k+1$}.
$$
This then implies that if $0 \leq p \leq \ell$, then
$$
K_{p,q}(\Sigma_k, \sO_{\Sigma_k}(1)) \neq 0~~\Longleftrightarrow~~\text{$p=0,q=0$ or $1 \leq p \leq \ell$ and $q=k+1$}.
$$
Thus the shape of the first $\ell$ steps of the minimal free resolution of $R(\Sigma_k, \sO_{\Sigma_k}(1))$ is completely determined. Recall from \cite[Theorem 1.2]{ENP1} that the Castelnuovo--Mumford regularity of $\sO_{\Sigma_k}$ is $2k+2$ if $g \geq 1$ and $k+1$ if $g=0$. As $\Sigma_k \subseteq \nP^r$ is arithmetically Cohen--Macaulay, the projective dimension of $R(\Sigma_k, \sO_{\Sigma_k}(1))$ is $e:= \codim \Sigma_k$. Notice that if $\deg L = 2g+2k+1+\ell$, then $\ell = e-g$. To summarize the discussion, the Betti table of $R(\Sigma_k, \sO_{\Sigma_k}(1))$ is the following:

\begin{table}[h]
\texttt{\begin{tabular}{c|ccccccccc}
& $0$ & $1$ & $2$ & $~\cdots$ &$e-g-1$ & $e-g$ &  $e-g+1$ & $\cdots$ & $~e$ \\ \hline
$0$ & 1 & - & - & $~\cdots$  & - & - & - & $\cdots$ & ~-  \\
$1$ & - & -  & -  & $~\cdots$ & - & - & - & $\cdots$ & ~- \\
$\vdots$ & $\vdots$  & $\vdots$  & $\vdots$ & & $\vdots$   & $\vdots$  & $\vdots$ & & $~\vdots$ \\
$k$ & - & - & - & $~\cdots$ & - & - & - & $\cdots$ & ~- \\
$k+1$ & - & * & * & $~\cdots$ & * & * & ? & $\cdots$ & ~? \\
$k+2$ & - & - & - & $~\cdots$ & - & -  & ? & $\cdots$ & ~? \\
$\vdots$ & $\vdots$  & $\vdots$  & $\vdots$ & & $\vdots$  & $\vdots$ &$\vdots$  & & $~\vdots$\\
$2k+2$ & - & - & - & $~\cdots$ & -  & - & ? & $\cdots$ & ~?
\end{tabular}}\\[10pt]
\caption{\label{table}The Betti table of $R(\Sigma_k, \sO_{\Sigma_k}(1))$}
\end{table}

\noindent \\[-30pt] Here ``\texttt{-}'' indicates a zero entry, ``\texttt{*}'' indicates a nonzero entry, and ``\texttt{?}'' indicates an entry not yet determined. We remark that the syzygies of $\Sigma_k$ in $\nP^r$ are not ``asymptotic syzygies'' considered in \cite{EL2} and \cite{Park}.

\medskip

It is natural to study the undetermined part of the Betti table of $R(\Sigma_k, \sO_{\Sigma_k}(1))$. The problem is to determine vanishing and nonvanishing of $K_{p,q}(\Sigma_k, \sO_{\Sigma_k}(1))$ for $e-g+1 \leq p \leq e$ and $k+1 \leq q \leq 2k+2$. The case of $g=0$ is a classical result, the case of $g=1$ was done by Graf von Bothmer--Hulek \cite{GvBH} and Fisher \cite{Fisher}, and the case of $g=2$ was recently settled by Li \cite{Li}. For the case of $g \geq 3$, we assume that $\deg L$ is sufficiently large.
Consider the case of $k=0$. Green--Lazarsfeld \cite[Theorem 2]{GL} proved that 
$$
K_{p,2}(C, L) \neq 0~~\text{ for $e-g+1 \leq p \leq e$}.
$$
This was recently generalized by Taylor \cite[Corollary 3.6]{Taylor} as
$$
K_{p,2k+2}(\Sigma_k, \sO_{\Sigma_k}(1)) \neq 0~~\text{ for $e-g+1 \leq p \leq e$}.
$$
The gonality conjecture of Green--Lazarsfeld asserts that
$$
K_{p,1}(C, L) \neq 0~~\text{ for $1 \leq p \leq\codim (C)-\gon(C) +1$},
$$
where 
$$
\gon(C):=\min\{ d \mid \text{$C$ carries a linear series $g_d^1$} \}
$$
is the \emph{gonality} of $C$. The gonality conjecture was established by Ein--Lazarsfeld \cite{EL3} and Rathmann \cite{Rathmann}. This suggests that the geometry of $C$ is deeply related to vanishing and nonvanishing of $K_{p,1}(C, L)$ for large $p$. Along this line, Lawrence Ein previously asked what kind of geometry of $C$ is involved in the behavior of vanishing and nonvanishing of $K_{p,q}(\Sigma_k, \sO_{\Sigma_k}(1))$ for large $p$. In \cite{CK}, the first and second authors proposed that one should consider the \emph{gonality sequence} of $C$. For an integer $q \geq 0$, let
$$
\gamma^q=\gamma^q(C):=\min\{d -q \mid \text{$C$ carries a linear series $g_d^q$} \}.
$$
Then $\gamma^1+1=\gon(C)$. The gonality sequence of $C$ is a sequence $(\gamma^0+0, \gamma^1+1, \gamma^2+2, \ldots)$. The gonality sequence was previously studied by several authors in the theory of algebraic curves (see e.g., \cite{CKM}, \cite{LM}). The conjecture of the first and second authors in \cite{CK} predicts that
$$
K_{p,k+1}(\Sigma_k, \sO_{\Sigma_k}(1)) \neq 0~~\Longleftrightarrow~~1\leq p \leq e-\gamma^{k+1}.
$$
However, the cases of $K_{p,q}(\Sigma_k, \sO_{\Sigma_k}(1))$ for $k+2 \leq q \leq 2k+1$ remained a mystery in general.

\medskip

In this paper, we completely resolve all the aforementioned problems at least when $L$ is sufficiently positive: We show that the gonality sequence of $C$ determines vanishing and nonvanishing of $K_{p,q}(\Sigma_k, \sO_{\Sigma_k}(1))$ for $e-g+1 \leq p \leq e$ and $k+1 \leq q \leq 2k+2$. 

\begin{Thm}\label{main}
Let $C$ be a smooth projective curve of genus $g \geq 2$, and $L$ be a very ample line bundle of sufficiently large degree on $C$. For an integer $k \geq 0$, consider the $k$-th secant variety $\Sigma_k$ of $C$ in $\nP H^0(C, L) = \nP^r$, and put $e:= \codim \Sigma_k = r-2k-1$. For each $k+1 \leq q \leq 2k+2$, if $e-g+1 \leq p \leq e$, then we have 
$$
K_{p,q}(\Sigma_k, \sO_{\Sigma_k}(1)) \neq 0~~\Longleftrightarrow~~e - g+1 \leq p \leq e - \gamma^{2k+2-q}(C).
$$
\end{Thm}

As we discussed before, together with \cite{ENP1}, our main theorem completely determines the shape of the Betti table of $R(\Sigma_k, \sO_{\Sigma_k}(1))$ in Table \ref{table}. Our approach using secant varieties gives an alternative proof of the gonality conjecture which is nothing but the case of $k=0$ in Theorem \ref{main}.
On the other hand, by duality, we have
$$
K_{p,q}(\Sigma_k, \sO_{\Sigma_k}(1)) = K_{e-p, 2k+2-q}(\Sigma_k, \omega_{\Sigma_k}; \sO_{\Sigma_k}(1))^{\vee}.
$$
The nontrivial parts  covered by Theorem \ref{main} are $K_{p,q}(\Sigma_k, \omega_{\Sigma_k}; \sO_{\Sigma_k}(1))$ for $0 \leq p \leq e-1$ and $0 \leq q \leq k+1$. The Betti table of $R(\Sigma_k, \omega_{\Sigma_k}; \sO_{\Sigma_k}(1))$ in this range -- the reverse of the part marked with ``\texttt{?}'' in Table \ref{table} -- is the following:

\begin{table}[h]
\begin{tabular}{c}
$\overbrace{\texttt{* $\cdots$ * * $\cdots$ * * $\cdots$ * * $\cdots$ *}}^{g}$\\
$\underbrace{\texttt{- $\cdots$ -}}_{\gamma^1(C)}$ \texttt{* $\cdots$ * * $\cdots$ * * $\cdots$ *}\\
$\underbrace{\texttt{- $\cdots$ - - $\cdots$ -}}_{\gamma^2(C)}$ \texttt{* $\cdots$ * * $\cdots$ *}\\
$\vdots$ \\[5pt]
$\underbrace{\texttt{- $\cdots$ - - $\cdots$ - - $\cdots$ -}}_{\gamma^{k+1}(C)}$ \texttt{* $\cdots$ *}
\end{tabular}\\[10pt]
\caption{\label{table2}The Betti table of $R(\Sigma_k, \omega_{\Sigma_k}; \sO_{\Sigma_k}(1))$}
\end{table}

\noindent \\[-30pt] Notice that $\gamma^q(C) = g$ for $q \geq g$. Thus if $k \geq g-1$, then the last $k-g+2$ rows of Table \ref{table2} are all vanishing; in particular,
\begin{align*}
&K_{p, k+1}(\Sigma_k, \sO_{\Sigma_k}(1)) =0~~\text{ for $e-g+1 \leq p \leq e$}; \\
&K_{p,q}(\Sigma_k, \sO_{\Sigma_k}(1)) = 0~~\text{ for all $p$ and $k+2 \leq q \leq 2k+2-g$}.
\end{align*}
Next, observe that the first $m+1$ rows of Table \ref{table2} have the same vanishing and nonvanishing patterns as those of the Betti table of $R(\Sigma_m, \omega_{\Sigma_m}; \sO_{\Sigma_m}(1))$ for each $0 \leq m \leq k$. Thus the syzygies of secant varieties of $C$ have a surprisingly uniform behavior governed by the gonality sequence of $C$. It resembles a matryoshka doll, which repeats similar patterns over and over again, so one may say that there is a ``matryoshka structure'' among secant varieties of $C$ in the sense of \cite{CK}.

\medskip

To prove Theorem \ref{main}, we utilize Bertram's construction \cite{Bertram} as in \cite{ENP1}. There is a vector bundle $E_{k+2, L}$ on the symmetric product $C_{k+2}$ of $C$ such that $\beta_{k+1} \colon \nP (E_{k+2, L}) \to \Sigma_{k+1}$ is a resolution of singularities and $Z_{k}:=\beta_{k+1}^{-1}(\Sigma_{k})$ is an effective divisor.  It is worth noting that we are working with $\Sigma_{k+1}$ to prove Theorem \ref{main} for $\Sigma_k$ instead of going to $\Sigma_{k-1}$ as in \cite{ENP1}.

\subsection*{Vanishing} Using the Du Bois-type condition
$$
R^i \beta_{k+1,*} \sO_{\nP (E_{k+2, L})}(-Z_{k}) = \begin{cases} \sI_{\Sigma_{k}|\Sigma_{k+1}} & \text{if $i=0$} \\ 0 & \text{if $i>0$}\end{cases}
$$
established in \cite{ENP1} and proceeding by induction on $q-k-1$, we reduce the vanishing part of Theorem \ref{main} to
$$
H^{q^*+1}(C_{k+2}, \A{p^*+q^*} M_{E_{k+2, L}} \otimes S_{k+2, \omega_C})=0,
$$
where $p^*:=e-p$ and $q^*:=2k+2-q$. Here $M_{E_{k+2, L}}$ is the kernel of the evaluation map $H^0(C, L) \otimes \sO_{C_{k+2}} \to E_{k+2, L}$, and $S_{k+2, \omega_C}$ is a line bundle with $q_{k+2}^* S_{k, L} = L^{\boxtimes k+2}$, where $q_{k+2} \colon C^{k+2} \to C_{k+2}$ is the map given by $(x_1, \ldots, x_{k+2}) \mapsto x_1 + \cdots + x_{k+2}$.
We then show that
\begin{align*}
&H^{q^*+1}(C_{k+2}, \A{p^*+q^*} M_{E_{k+2, L}} \otimes S_{k+2, \omega_C}) \\
&= H^{q^*+1}(C_{p^*+q^*} \times C_{k+2}, (N_{p^*+q^*,L} \boxtimes S_{k+2, \omega_C})(-D_{p^*+q^*, k+2})).
\end{align*}
Here $N_{p^*+q^*, L}$ is a line bundle with $q_{p^*+q^*}^* N_{p^*+q^*, L} = L^{\boxtimes p^*+q^*}(-\Delta)$, where $\Delta$ is the sum of all pairwise diagonal on $C^{p^*+q^*}$. 
Let  $\pr_1 \colon C_{p^*+q^*} \times C_{k+2} \to C_{p^*+q^*}$ be the projection map, and $D_{p^*+q^*, k+2}:=\{(\xi_1+x, \xi_2+x) \mid \xi_1 \in C_{p^*+q^*-1}, \xi_2 \in C_{k+1}, x \in C\}$ be an effective divisor on $C_{p^*+q^*} \times C_{k+2}$.
Then it is enough to check that
\begin{equation}\label{eq:vanintro}\tag{$\star$}
H^i(C_{p^*+q^*}, R^{q^*+1-i} \pr_{1,*} (N_{p^*+q^*,L} \boxtimes S_{k+2, \omega_C})(-D_{p^*+q^*, k+2})) =0~~\text{ for $0 \leq i \leq q^*+1$}.
\end{equation}
When $i>0$, the cohomology vanishing (\ref{eq:vanintro}) follows from Fujita--Serre vanishing since $N_{p^*+q^*, L}$ is sufficiently positive. When $i=0$, the fiber of $R^{q^*+1} \pr_{1,*} (N_{p^*+q^*,L} \boxtimes S_{k+2, \omega_C})(-D_{p^*+q^*,k+2})$ over $\xi \in C_{p^*+q^*}$ is 
$$
H^{q^*+1}(C_{k+2}, S_{k+2, \omega_C(-\xi)}) = S^{k+1-q^*} H^0(C, \omega_C(-\xi)) \otimes \A{q^*+1} H^1(C, \omega_C(-\xi)).
$$
However, $h^1(C, \omega_C(-\xi)) \leq q^*$ thanks to the ``gonality sequence condition'' $\gamma^{q^*}(C) \geq p^*+1$, so 
$$
R^{q^*+1} \pr_{1,*} (N_{p^*+q^*,L} \boxtimes S_{k+2, \omega_C})(-D_{p^*+q^*,k+2}) = 0.
$$
Thus (\ref{eq:vanintro}) holds for $i=0$.

\subsection*{Nonvanishing}
For the nonvanishing part of Theorem \ref{main}, it suffices to see that the map
$$
H^{q-1}(\Sigma_k, \A{p+q-1}M_{\sO_{\Sigma_k}(1)}\otimes \sI_{\Sigma_{k-1}|\Sigma_k}(1)) \longrightarrow H^q(\Sigma_{k+1}, \A{p+q-1}M_{\sO_{\Sigma_{k+1}}(1)}\otimes \sI_{\Sigma_k|\Sigma_{k+1}}(1))
$$
is nonzero. Arguing as in the proof of the vanishing part, we reduce the problem to showing that the map
\begin{multline}\label{eq:nonvanintro}\tag{{\tiny $\blacklozenge$}}
R^{q^*+1} \pr_{1,*} (N_{p^*+q^*,L} \boxtimes S_{k+2, \omega_C})(-D_{p^*+q^*,k+2}) \\
 \longrightarrow R^{q^*} \pr_{1,*} (N_{p^*+q^*,L} \boxtimes S_{k+1, \omega_C})(-D_{p^*+q^*,k+1}) \otimes H^1(C, \omega_C)
\end{multline}
is nonzero. By the ``gonality sequence condition'' $\gamma^{q^*}(C) \leq p^*$, we can find an effective divisor $\xi \in C_{p^*+q^*}$ with $h^0(C, \omega_C(-\xi)) = g-p^*$ and $h^1(C, \omega_C(-\xi)) = q^*+1$. The map (\ref{eq:nonvanintro}) looks like $\id_{S^{k+1-q^*} H^0(C, \omega_C(-\xi))} \otimes \delta$ fiberwisely over $\xi \in C_{p^*+q^*}$, where $\delta$ is the Koszul-like map
$$
\A{q^*+1} H^1(C, \omega_C(-\xi)) \longrightarrow \A{q^*} H^1(C, \omega_C(-\xi)) \otimes H^1(C, \omega_C).
$$
Since $\delta$ is clearly nonzero, it follows that the map (\ref{eq:nonvanintro}) is nonzero.

\medskip

The paper is organized as follows. We begin with collecting basic relevant facts on the gonality sequence of a curve in Section \ref{secgon}. Section \ref{secsym} provides a review of basic properties of symmetric products and secant varieties of curves. Section \ref{secproof} is devoted to the proof of Theorem \ref{main}. Finally, in Section \ref{secquestions}, we present some complementary results, and we also discuss some open problems.

\subsection*{Acknowledgments.} We would like to thank Lawrence Ein and Wenbo Niu for valuable and interesting discussions. We are also grateful to Daniele Agostini, Marian Aprodu, Daniel Erman, Robert Lazarsfeld, Frank-Olaf Schreyer, Jessica Sidman for their interests.

%%%%%%%%%%%%%%%%%%%%%%%%%%%%%%%%%%%%%%%%%%%%%%%%%%%%%%%%
\section{Gonality Sequences}\label{secgon}
%%%%%%%%%%%%%%%%%%%%%%%%%%%%%%%%%%%%%%%%%%%%%%%%%%%%%%%%
In this section, we recall the definition and basic properties of the gonality sequence of a smooth projective curve $C$ of genus $g \geq 2$, and we show some relevant facts.

\begin{Def}\label{defhighergon}
For any integer $q\geq 0$, we define
$$
\gamma^q(C):=\min\{ d- q \mid \textup{$C$ carries a linear series $g_d^q$}\}.
$$
A sequence $(\gamma^0(C)+0, \gamma^1(C)+1, \gamma^2(C)+2, \ldots)$ is called the \emph{gonality sequence} of $C$.
\end{Def}

Note that $\gamma^0(C) = 0$ and $\gamma^1(C)+1 = \gon (C)$ is the gonality of $C$. The following is an easy consequence of the Riemann--Roch theorem, the Clifford theorem, and Brill--Noether theory.

\begin{Lem}[{\cite[Lemmas 3.1 and 3.2]{LM}}]\label{basicgonseq}
We have the following:
\begin{enumerate}
\item $\gamma^q(C) \leq \gamma^{q+1}(C)$ for $q \geq 0$.
\item $\min \{q, g\} \leq \gamma^q(C) \leq g - \lfloor g/(q+1) \rfloor$ for $q \geq 0$. In particular, $\gamma^{g-1}(C) = g-1$ and $\gamma^q(C) = g$ for $q \geq g$.
\end{enumerate}
\end{Lem}

If $C$ is hyperelliptic, then $\gamma^q(C) = q$ for $q \leq g$. However, as was remarked in \cite{LM}, it is not easy to compute the gonality sequence of a curve in general. We refer to \cite{LM} for more details.

\medskip

Next, we introduce a new positivity notion for a line bundle on $C$.

\begin{Def}
Let $B$ be a line bundle on $C$. For integers $w, p \geq 0$, we say that $B$ is \emph{$w$-weakly $p$-very ample} if 
$$
\corank \big(H^0(C,B) \longrightarrow H^0(C,B|_\xi) \big)\leq w
$$
for every effective divisor $\xi$ of degree $p+w+1$ on $C$.
\end{Def}

Note that $B$ is $0$-weakly $p$-very ample if and only if $B$ is $p$-very ample. Recall that $\gamma^1(C) \geq p+1$ (i.e., $\gon(C) \geq p+2$) if and only if $\omega_C$ is $p$-very ample. The next proposition is a generalization of this fact.

\begin{Prop}\label{gonseq-h^1}
Let $q \geq 1$ be an integer. Then the following are equivalent:
\begin{enumerate}
\item $\gamma^q(C) \geq p+1$.
\item $h^0(C, \sO_C(\xi)) = h^1(C, \omega_C(-\xi)) \leq q$ for every effective divisor $\xi$ of degree $p+q$ on $C$.
\item $\omega_C$ is $(q-1)$-weakly $p$-very ample.
\end{enumerate} 
In particular,
\begin{align*}
\gamma^q(C)&=\max\{p\geq0 \mid \omega_C\textup{ is $(q-1)$-weakly $p$-very ample}\}+1 \\
&=\min\{p\geq0 \mid \omega_C\textup{ fails to be $(q-1)$-weakly $p$-very ample}\}.
\end{align*}
\end{Prop}

\begin{proof}
It is clear from the definitions.
\end{proof}

\begin{Lem}\label{exactdivisor}
If $\omega_C$ is not $w$-weakly $p$-very ample with $0\leq p\leq g$, then there is an effective divisor $\xi$ of degree $p+w+1$ such that
$$
\corank \big(H^0(C,\omega_C) \longrightarrow H^0(C,\omega_C|_\xi) \big)=w+1,
$$
i.e., $h^0(C, \omega_C(-\xi))=g-p$ and $h^1(C,\omega_C(-\xi))=w+2$.
\end{Lem}

\begin{proof}
Since $\omega_C$ is not $w$-weakly $p$-very ample, there is an effective divisor $\xi_0$ of degree $p+w+1$  on $C$ such that
$$
h^1(C, \omega_C(-\xi_0))\geq w+2.
$$
If $h^1(C, \omega_C(-\xi_0))=w+2$, then we are done by taking $\xi=\xi_0$.
Suppose that $h^1(C, \omega_C(-\xi_0))\geq w+3$. The Riemann--Roch theorem yields
$$
h^0(C, \omega_C(-\xi_0))= g-p+h^1(C, \omega_C(-\xi_0)) - w-2 \geq 1.
$$
It is elementary to see that if $B$ is a line bundle on $C$ with $H^0(C, B)\neq0$ and $H^1(C, B)\neq0$, then
$$
h^0(C, B(-x_0+x_1))=h^0(C, B)-1~~\text{ and }~~h^1(C, B(-x_0+x_1))=h^1(C, B)-1
$$
for general points $x_0, x_1 \in C$. Thus we find
$$
h^0(C, \sO_C(\xi_0 + x_0 - x_1)) = h^1(C, \omega_C(-\xi_0 - x_0 + x_1)) = h^1(C, \omega_C(-\xi_0)) -1 \geq w+2,
$$
so we can choose an effective divisor $\xi_1\in|\xi_0+x_0-x_1|$ of degree $p+w+1$. Then
$$
h^1(C, \omega_C(-\xi_1)) = h^1(C, \omega_C(-\xi_0)) -1.
$$ 
Continuing this process, we finally reach an effective divisor $\xi$ of degree $p+w+1$ such that
$h^1(C, \omega_C(-\xi))=w+2$.
\end{proof}

%%%%%%%%%%%%%%%%%%%%%%%%%%%%%%%%%%%%%%%%%%%%%%%%%%%%%%%%
\section{Symmetric Products and Secant Varieties of Curves}\label{secsym}
%%%%%%%%%%%%%%%%%%%%%%%%%%%%%%%%%%%%%%%%%%%%%%%%%%%%%%%%

In this section, we review basic properties of symmetric products and secant varieties of smooth projective curves, and we show some useful lemmas for the proof of Theorem \ref{main}. We refer to \cite{Bertram} and \cite{ENP1} for a more detailed account.

\medskip

Let $C$ be a smooth projective curve of genus $g$. For an integer $k \geq 1$, we write the $k$-th symmetric product of the curve $C$ as $C_k$ and the $k$-th ordinary product of the curve $C$ as $C^k$. The symmetric group $\mathfrak{S}_k$ naturally acts on $C^k$, and $C_k=C^k/\mathfrak{S}_k$. We have the quotient morphism
$$
q_k \colon C^k \longrightarrow C_k,~~(x_1, \ldots, x_k) \longmapsto x_1 + \cdots + x_k,
$$
which is a finite flat surjective morphism of degree $k!$. For a line bundle $L$ on $C$, there are two line bundles $S_{k,L}$ and $N_{k,L}$ on $C_k$ such that 
$$
q_k^* S_{k,L} = L^{\boxtimes k} = \underbrace{L \boxtimes \cdots \boxtimes L}_{k~\text{times}}~~\text{ and }~~q_k^* N_{k,L} = L^{\boxtimes k}\big( -\sum_{1 \leq i < j \leq k} \Delta_{i,j} \big),
$$ 
where $\Delta_{i,j}:=\{(x_1, \ldots, x_k) \in C^k \mid x_i = x_j\}$ is a pairwise diagonal. Let $\delta_k$ be a divisor on $C_k$ such that $\sO_{C_k}(\delta_k) = S_{k,\sO_C} \otimes N_{k,\sO_C}^{-1}$. Then $N_{k,L} = S_{k,L}(-\delta_k)$ for any line bundle $L$ on $C$. It is well known that
$$
H^0(C_k, S_{k,L}) = S^k H^0(C, L)~~\text{ and }~~H^0(C_k, N_{k,L}) = \A{k} H^0(C, L).
$$
Furthermore, we have the following.

\begin{Lem}[{\cite[Lemma 2.4]{Agostini}, \cite[Lemma 3.7]{ENP1}}]\label{H^iS_k,N_k}
We have
\begin{align*}
&H^i(C_k,S_{k,L})=S^{k-i}H^0(C,L)\otimes\A{i}H^1(C,L)~~\text{ for $i \geq 0$};\\
&H^i(C_k,N_{k,L})=\A{k-i}H^0(C,L)\otimes S^i H^1(C,L)~~\text{ for $i \geq 0$.}
\end{align*}
\end{Lem}

Let $D_{k,m}$ be the effective divisor on $C_k \times C_m$ given by the image of the map
$$
C_{k-1} \times C \times C_{m-1} \longrightarrow C_k \times C_m,~~(\xi_1, x, \xi_2) \longmapsto (\xi_1+x, \xi_2+x),
$$
and $\pr_1 \colon C_k \times C_m \to C_k,~\pr_2 \colon C_k \times C_m \to C_m$ be the projections. 
For $\xi \in C_k$, we set $C_{m, \xi}:=\pr_1^{-1}(\xi) = \{ \xi \} \times C_m$. Then $\sO_{C_m}(C_{m, \xi} \cap D_{k,m}) = S_{m, \sO_C(\xi)}$.
For a coherent sheaf $\sF$ on $C_m$, we put
$$
M_k^i \sF:=R^i \pr_{1,*} (\sO_{C_k} \boxtimes \sF )(-D_{k,m})~~\text{ for $i \geq 0$}.
$$
It is a coherent sheaf on $C_k$. Identifying $C_{m,\xi}=C_m$ for each $\xi \in C_k$, we have a natural map
$$
\rho^i(\xi) \colon M_k^i \sF \otimes \kk(\xi) \longrightarrow H^i(C_m, \sF \otimes S_{m, \sO_C(-\xi)}).
$$
Suppose that $\sF$ is flat over $C_k$. By Grauert's theorem, when $h^i(C_m, \sF \otimes S_{m, \sO_C(-\xi)})$ is constant for all $\xi \in C_k$, $M_k^i \sF$ is a vector bundle and $\rho^i(\xi)$ is an isomorphism . 
By the cohomology and base change, when $\rho^{i+1}(\xi)$ is surjective for $\xi \in C_k$, $\rho^i(\xi)$ is an isomorphism if and only if $M_k^{i+1} \sF$ is locally free in a neighborhood of $\xi \in C_k$.

\begin{Lem}[{cf. \cite[Lemma 1.2]{EL3}}]\label{Fujita}
For a given coherent sheaf $\sF$ on $C_m$, if $\deg L$ is sufficiently large, then
$$
H^i(C_k, M_k^j \sF \otimes N_{k,L})=0~~\text{ for $i>0$ and $j \geq 0$}.
$$
\end{Lem}

\begin{proof}
As $\deg L \gg 0$, we may write $L=L' \otimes \sO_C(mx)$ for a point $x \in C$ and a sufficiently large integer $m \gg 0$ such that $N_{k,L'}$ is nef. Then $N_{k,L}=N_{k,L'} \otimes S_{k, \sO_C(x)}^{m}$. Since $S_{k, \sO_C(x)}$ is ample and $m$ is sufficiently large, the lemma follows from Fujita--Serre vanishing \cite[Theorem 1.4.35]{positivity}.
\end{proof}

In the above situation, we now consider the case $m=1$. Then $D_{k,1}$ is the image of the injective map
$$
C_{k-1} \times C \longrightarrow C_k \times C,~~(\xi, x) \longmapsto (\xi+x, x).
$$
Let $\sigma_k:=\pr_1|_{D_{k,1}}$. Identifying $D_{k,1}$ with $C_{k-1} \times C$, we obtain a map
$$
\sigma_{k} \colon C_{k-1} \times C \longrightarrow C_{k},~~(\xi, x) \longmapsto \xi+x,
$$
which is a finite flat surjective morphism of degree $k$. If we view $C_{k}$ as the Hilbert scheme of $k$ points on $C$, then $\sigma_k$ is the universal family.
The \emph{tautological bundle} on $C_k$ associated to $L$ is defined to be
$$
E_{k,L} := \sigma_{k,*} (\sO_{C_{k-1}} \boxtimes L). 
$$
It is a vector bundle of rank $k$ on $C_k$. Note that $H^0(C, E_{k,L}) = H^0(C, L)$ and $\det E_{k,L} = N_{k,L}$. Suppose that $L$ is $(k-1)$-very ample. Then $E_{k,L}$ is globally generated. Applying $\pr_{1,*}$ to the short exact sequence
$$
\begin{tikzcd}
0 \ar[r]  & (\sO_{C_k} \boxtimes L)(-D_{k,1}) \ar[r,"\cdot D_{k,1}"] & \sO_{C_k} \boxtimes L  \ar[r] & \sO_{C_{k-1}} \boxtimes L \ar[r] & 0,
\end{tikzcd}
$$
we get a short exact sequence
$$
\begin{tikzcd}
0 \ar[r]  & M_{E_{k,L}} \ar[r] & H^0(C, L) \otimes \sO_{C_k}  \ar[r,"\operatorname{ev}"] & E_{k,L} \ar[r] & 0.
\end{tikzcd}
$$
Notice that $M_{E_{k,L}} = M_k^0 L$ is a vector bundle of rank $h^0(C, L) - k$ on $C_k$. This short exact sequence looks like
$$
\begin{tikzcd}
0 \ar[r]  & H^0(C, L(-\xi)) \ar[r] & H^0(C, L)  \ar[r] & H^0(C, L|_{\xi}) \ar[r] & 0
\end{tikzcd}
$$
over $\xi \in C_k$ fiberwisely.

\begin{Lem}\label{wedgeM_{E_{k,L}}}
Suppose that $\deg L \geq 2g+k-1$. Then
$$
M^i_k N_{m,L} = \begin{cases} \A{m} M_{E_{k,L}} & \text{if $i=0$} \\ 0 & \text{if $i>0$}. \end{cases}
$$
In particular, for any line bundle $B$ on $C_k$, we have
$$
H^i(C_k \times C_m, (B \boxtimes N_{m, L})(-D_{k,m})) = H^i(C_k, \A{m} M_{E_{k,L}} \otimes B)~~\text{ for $i\geq 0$}.
$$
\end{Lem}

\begin{proof}
By Lemma \ref{H^iS_k,N_k},
$$
H^i(C_m, N_{m, L(-\xi)}) = \A{m-i} H^0(C, L(-\xi)) \otimes S^i H^1(C, L(-\xi)).
$$
for any $\xi \in C_k$. Since $\deg L(-\xi) \geq 2g-1$, it follows that $H^1(C, L(-\xi))=0$. Thus we get $H^i(C_m, N_{m, L(-\xi)}) =0$ for $i>0$, so we obtain $M^i_k N_{m,L} = 0$ for $i>0$. Note that $M^0_k N_{m,L}$ is a vector bundle on $C_k$ whose fiber is $\A{m} H^0(C, L(-\xi))$ over $\xi \in C_k$. Applying $\pr_{1,*}$ to the injective map
$$
(\sO_{C_k} \boxtimes N_{m,L})(-D_{k,m}) \longhookrightarrow \sO_{C_k} \boxtimes N_{m,L},
$$
we get an injective map
$$
M^0_k N_{m,L} \longhookrightarrow \A{m} H^0(C, L) \otimes \sO_{C_k},
$$
which looks like
$$
\A{m} H^0(C, L(-\xi)) \longhookrightarrow \A{m} H^0(C, L)
$$
over $\xi \in C_k$ fiberwisely. On the other hand, notice that $L$ is $(k-1)$-very ample. The injective map 
$$
M_{E_{k,L}} \longhookrightarrow H^0(C, L) \otimes \sO_{C_k}
$$ 
induces an injective map
$$
\A{m} M_{E_{k,L}}  \longhookrightarrow \A{m} H^0(C, L) \otimes \sO_{C_k},
$$
which looks like
$$
\A{m} H^0(C, L(-\xi)) \longhookrightarrow \A{m} H^0(C, L)
$$
over $\xi \in C_k$ fiberwisely. Thus we can conclude that $M^0_k N_{m,L} = \A{m} M_{E_{k,L}}$. Now, the second statement follows from the projection formula and the Leray spectral sequence for $\pr_1$. 
\end{proof}

From now on, as in \cite{Bertram} and \cite{ENP1}, suppose that 
$$
\deg L \geq 2g+2k+1.
$$
For an integer $k \geq 0$, let
$$
B_k=B_k(L):=\nP (E_{k+1, L})
$$
with the canonical projection $\pi_k \colon B_k \to C_{k+1}$, and $H_k$ be a tautological divisor so that $\sO_{B_k}(H_k) = \sO_{\nP (E_{k+1, L})}(1)$. As $E_{k+1, L}$ is globally generated, $H_k$ is base point free. Note that 
$$
H^0(B_k, H_k) = H^0(C_{k+1}, E_{k+1, L}) = H^0(C, L).
$$
The image of the morphism given by the complete linear system $|H_k|$ is the \emph{$k$-th secant variety} $\Sigma_k$ of $C$ in $\nP H^0(C, L)= \nP^r$. Denote the induced map by
$$
\beta_k \colon B_k \longrightarrow \Sigma_k,
$$
which is a resolution of singularities. By \cite[Theorem 1.1]{ENP1}, $\Sigma_k$ has normal Du Bois singularities, and in particular,
$$
\beta_{k,*} \sO_{B_k} = \sO_{\Sigma_k}.
$$ 
By \cite[Theorem 1.2]{ENP2}, $\Sigma_k \subseteq \nP^r$ is arithmetically Cohen--Macaulay, and $H^{2k+1}(\Sigma_k, \sO_{\Sigma_k}(m))=0$ for $m>0$.
Note that $\sO_{B_k}(H_k) = \beta^* \sO_{\Sigma_k}(1)$. Put $M_{H_k}:=\beta_k^* M_{\sO_{\Sigma_k}(1)}$, which fits into a short exact sequence
$$
\begin{tikzcd}
0 \ar[r]  & M_{H_k} \ar[r] & H^0(C, L) \otimes \sO_{B_k} \ar[r,"\operatorname{ev}"] & \sO_{B_k}(H_k) \ar[r] & 0.
\end{tikzcd}
$$
Set $Z_{k-1}:=\beta_k^{-1}(\Sigma_{k-1})$, which is an irreducible effective divisor on $B_k$. Then $\beta_{k,*} \sO_{B_k}(-Z_{k-1}) = \sI_{\Sigma_{k-1}|\Sigma_k}$. We have a commutative diagram
$$
\begin{tikzcd}[column sep=1.5cm] 
			& B_{k} \ar[dl,"\pi_k", swap] \ar[dr,"\beta_{k}"] 	& Z_{k-1} \ar[l, hook'] \ar[dr] \\
C_{k+1} 	& 										& \Sigma_{k} 							& \Sigma_{k-1}. \ar[l, hook'] 	
\end{tikzcd}
$$
Notice that $\omega_{C_{k+1}} = N_{k+1, \omega_C}$. Then we have
$$
\omega_{B_k} = \sO_{B_k}(-(k+1)H_k) \otimes \pi_k^* (\omega_{C_{k+1}} \otimes \det E_{k+1, L}) = \sO_{B_k}(-(k+1)H_k) \otimes \pi_k^* S_{k+1, \omega_C \otimes L}(-2\delta_{k+1}).
$$
We will compute $\omega_{\Sigma_k}$ in Proposition \ref{Ein}. 
On the other hand, the map $\sigma_{k+1} \colon C_k \times C \to C_{k+1}$ provides a morphism $\alpha_{k} \colon B_{k-1} \times C \to B_k$ birational onto its image (see \cite[p.\ 432]{Bertram}). By \cite[Lemma 1.1 $(a)$]{Bertram} (see \cite[Subsection 3.2]{ENP1}), we have a commutative diagram
\begin{equation}\label{additionmapBertram}
\begin{tikzcd}
B_{k-1}\times C \ar[r,"\alpha_{k}"] \ar[d] & B_{k} \ar[d,"\beta_{k}"] \\
\Sigma_{k-1} \ar[r,hook] & \Sigma_{k},
\end{tikzcd}
\end{equation}
where the left vertical map is the first projection followed by $\beta_{k-1}$.

\begin{Prop}[{\cite{ENP1}}]\label{basicpropsecantvar}
We have the following:
\begin{enumerate}
 \item $\sO_{B_k}(Z_{k-1}) = \sO_{B_k}((k+1)H_k) \otimes \pi_k^* S_{k+1,L}(-2\delta_{k+1})^{-1}$ and $\omega_{B_k}(Z_{k-1}) = \pi_k^* S_{k+1, \omega_C}$. 
 \item $R^i \beta_{k,*} \sO_{B_k}(-Z_{k-1}) = \begin{cases} \sI_{\Sigma_{k-1}|\Sigma_k} & \text{if $i=0$} \\ 0 & \text{if $i>0$}.\end{cases}$
 \item $R^i \pi_{k,*} \A{j} M_{H_k} = \begin{cases} \A{j} M_{E_{k+1,L}} & \text{if $i=0$} \\ 0 & \text{if $i>0$}. \end{cases}$
\end{enumerate}
\end{Prop}

\begin{proof}
$(1)$ The first assertion is \cite[Proposition 3.5 $(2)$]{ENP1}. Note that $\det E_{k+1, L} = S_{k+1, L}(-\delta_{k+1})$ and $\omega_{C_{k+1}} = S_{k+1, \omega_C}(-\delta_{k+1})$.  Since $\omega_{B_k} = \sO_{B_k}(-(k+1)H_{k+1})\otimes \pi_k^* S_{k+1, \omega_C \otimes L}(-2\delta_{k+1})$, the second assertion follows.

\medskip

\noindent $(2)$ It is \cite[Theorem 5.2 $(2)$]{ENP1}.

\medskip

\noindent $(3)$ It is shown in \cite[Proof of Lemma 5.1]{ENP1}. For reader's convenience, we give a sketch of the proof. We have a short exact sequence
$$
\begin{tikzcd}
0 \ar[r]  & \pi_k^* M_{E_{k+1, L}} \ar[r] &  M_{H_k}  \ar[r] & K \ar[r] & 0,
\end{tikzcd}
$$
where $K|_{\pi_k^{-1}(\xi)} = M_{\sO_{\nP^{k}}(1)}$ for $\xi \in C_{k+1}$. By Bott vanishing, 
$$
R^i \pi_{k,*} \A{j} K = \begin{cases} \sO_{C_{k+1}} & \text{if $i=0$ and $j=0$}\\ 0 & \text{if $i>0$ or $j>0$}. \end{cases}
$$
Considering the filtration of $\A{j} M_{H_k}$ associated to the above short exact sequence, we obtain the assertion.
\end{proof}

\begin{Lem}\label{duality}
We have
$$
\begin{array}{l}
H^i(\Sigma_k, \A{j} M_{\sO_{\Sigma_k}(1)} \otimes \sI_{\Sigma_{k-1}|\Sigma_k}(1))\\[5pt]
= H^i(B_k, \A{j} M_{H_k} \otimes \sO_{B_k}(H_k - Z_{k-1})) \\[5pt]
= H^{2k+1-i} (B_k, \A{r-j} M_{H_k} \otimes \omega_{B_k}(Z_{k-1}))^{\vee} \\[5pt]
= H^{2k+1-i} (C_{k+1}, \A{r-j} M_{E_{k+1, L}} \otimes S_{k+1, \omega_C})^{\vee} \\[5pt]
= H^{2k+1-i} (C_{r-j} \times C_{k+1}, (N_{r-j, L} \boxtimes S_{k+1, \omega_C})(-D_{r-j, k+1}))^{\vee}.
\end{array}
$$
\end{Lem}

\begin{proof}
The first equality follows from Proposition \ref{basicpropsecantvar} $(2)$ and the projection formula. Note that $\rank M_{H_k} = h^0(C, L)-1 = r$ and $\det M_{H_k} = \sO_{B_k}(-H_k)$. It follows that $\A{j} M_{H_k}^{\vee} = \A{r-j} M_{H_k} \otimes \sO_{B_k}(H_k)$. Then the second equality follows from Serre duality. The third equality follows from Proposition \ref{basicpropsecantvar} $(1)$ and $(3)$. The final equality follows from Lemma \ref{wedgeM_{E_{k,L}}}.
\end{proof}

Finally, we show some useful facts on the dualizing sheaf $\omega_{\Sigma_k}$. The following proposition will not be used in the proof of Theorem \ref{main} but will be used for some additional results. 

\begin{Prop}[Ein\footnote{This was shown to the third author by Lawrence Ein in personal communication.}]\label{Ein}
We have the following:
\begin{enumerate}
\item $\beta_{k,*}\omega_{B_k}(Z_{k-1}) = \omega_{\Sigma_k}$.
\item There is a short exact sequence
$$
\begin{tikzcd}
0 \ar[r]  & \beta_{k,*}\omega_{B_k}  \ar[r] & \omega_{\Sigma_k}  \ar[r] &\beta_{k,*}\omega_{Z_{k-1}}  \ar[r] & 0.
\end{tikzcd}
$$
\item $H^0(\Sigma_k, \omega_{\Sigma_k}(\ell)) = H^0(C_{k+1}, S^{\ell} E_{k+1, L} \otimes S_{k+1, \omega_C})$ for all $\ell \geq 0$.
\item If $k \geq 2$, then $H^0(\Sigma_k, \omega_{\Sigma_k}(\ell)) = H^0(\Sigma_{k-1}, \beta_{k,*}\omega_{Z_{k-1}}(\ell))$ for each $0 \leq \ell \leq k$.
\end{enumerate}
\end{Prop}

\begin{proof}
$(1)$ By \cite[Proposition 3.15]{ENP1}, there is a log resolution $b_k \colon \operatorname{bl}_k (B_k) \to B_k$ of $(B_k, Z_{k-1})$ constructed by Bertram in \cite{Bertram} such that
$$
b_k^*\omega_{B_k}(Z_{k-1})=\omega_{\operatorname{bl}_k(B_k)}(E_0+E_1+\cdots+E_{k-1}),
$$
where $E_0, E_1, \ldots, E_{k-2}$ are $\operatorname{bl}_k$-exceptional divisors and $E_{k-1} = \operatorname{bl}_{k,*}^{-1} Z_{k-1}$. We have
$$
(\beta_k \circ b_k)_*\omega_{\operatorname{bl}_k(B_k)}(E_0+E_1+\cdots+E_{k-1})=\beta_{k,*} \omega_{B_k}(Z_{k-1}).
$$
Note that $\beta_k \circ b_k \colon \operatorname{bl}_k(B_k) \to \Sigma_k$ is a log resolution of $\Sigma_k$.
Since $\Sigma_k$ has normal Cohen-Macaulay Du Bois singularities by \cite[Theorems 1.1 and 1.2]{ENP1}, it follows from \cite[Theorem 1.1]{KSS} that 
$$
(\beta_k \circ b_k)_*\omega_{\operatorname{bl}_k(B_k)}(E_0+E_1+\cdots+E_{k-1}) = \omega_{\Sigma_k}.
$$
Thus $\beta_{k,*}\omega_{B_k}(Z_{k-1}) = \omega_{\Sigma_k}$.

\medskip

\noindent $(2)$ We have a short exact sequence
$$
\begin{tikzcd}
0 \ar[r]  &\omega_{B_k}  \ar[r] & \omega_{B_k}(Z_{k-1})  \ar[r] &  \omega_{Z_{k-1}}   \ar[r] & 0.
\end{tikzcd}
$$
By Grauert--Riemenschneider vanishing, $R^i \beta_{k,*} \omega_{B_k}=0$ for $i > 0$.
Applying $\beta_{k,*}$ to the above short exact sequence, we obtain the assertion $(2)$.

\medskip

\noindent $(3)$ We have $H^0(\Sigma_k, \omega_{\Sigma_k}(\ell)) = H^0(B_k, \omega_{B_k}(Z_{k-1}+\ell H_k))$. Recall from Proposition \ref{basicpropsecantvar} $(1)$ that $\omega_{B_k}(Z_{k-1}) = \pi_k^* S_{k+1, \omega_C}$. Thus $H^0(B_k, \omega_{B_k}(Z_{k-1}+\ell H_k)) = H^0(C_{k+1}, S^{\ell} E_{k+1,L} \otimes S_{k+1, \omega_C})$.

\medskip

\noindent $(4)$ Since $R^i \beta_{k,*} \omega_{B_k}=0$ for $i > 0$, we have 
$$
H^i(\Sigma_k, \beta_{k,*}\omega_{B_k}(\ell))=H^i(B_k, \omega_{B_k}(\ell H_k)) ~~\text{ for $i \geq 0$.}
$$
If $k \geq 2$ and $0 \leq \ell \leq k$, we have $H^i(B_k, \omega_{B_k}(\ell H_k))=0$ for each $i=0,1$. Thus the assertion $(4)$ follows.
\end{proof}

%%%%%%%%%%%%%%%%%%%%%%%%%%%%%%%%%%%%%%%%%%%%%%%%%%%%%%%%
\section{Proof of Main Theorem}\label{secproof}
%%%%%%%%%%%%%%%%%%%%%%%%%%%%%%%%%%%%%%%%%%%%%%%%%%%%%%%%

In this section, we prove Theorem \ref{main}. First, we recall the setting. Let $C$ be a smooth projective curve of genus $g \geq 2$, and $L$ be a very ample line bundle on $C$. Consider the $k$-th secant variety $\Sigma_k$ of $C$ in $\nP H^0(C, L)=\nP^r$. Assume that $\deg L \gg 0$. When $k=0$ (i.e., $\Sigma_0=C$), Theorem \ref{main} is the gonality conjecture established by Ein--Lazarsfeld \cite{EL3} and Rathmann \cite{Rathmann}. Thus we assume that $k \geq 1$.\footnote{By a small modification, our proof works for the case of $k=0$. The vanishing part gives an alternative proof of the gonality conjecture. Indeed, when $k=0$ and $q=1$, we only need to verify (\ref{vanishing3}).} Put $e:=\codim \Sigma_k = r-2k-1$ and $\gamma^i:=\gamma^i(C)$ for $i \geq 0$. Fix an index $k+1 \leq q \leq 2k+2$.
%However, it is quite easy to check $K_{p,1}(C, L) \neq 0$ for $1 \leq p \leq\codim C-\gon(C)+1$. See \cite[Footnote 1]{EL3}.
%%%%%%%%%%%%%%%%%%%%%%%%%%%%%%%%%%%%%%%%%%%%%%%%%%%%%%%%
\subsection*{Vanishing}
%%%%%%%%%%%%%%%%%%%%%%%%%%%%%%%%%%%%%%%%%%%%%%%%%%%%%%%%
We show that 
\begin{equation}\label{vanishing1}
K_{p,q}(\Sigma_k, \sO_{\Sigma_k}(1)) = H^{q-1}(\Sigma_k, \A{p+q-1} M_{\sO_{\Sigma_k}(1)} \otimes \sO_{\Sigma_k}(1)) = 0~~\text{ for $p \geq e - \gamma^{2k+2-q}+1$}.
\end{equation}
Consider a short exact sequence
$$
\begin{tikzcd}
0 \ar[r]  & \sI_{\Sigma_k|\Sigma_{k+1}} \ar[r] & \sO_{\Sigma_{k+1}} \ar[r] & \sO_{\Sigma_k} \ar[r] & 0.
\end{tikzcd}
$$
This induces an exact sequence
\begin{multline}\label{exseqforvan}
H^{q-1}(\Sigma_{k+1}, \A{p+q-1}M_{\sO_{\Sigma_{k+1}}(1)} \otimes \sO_{\Sigma_{k+1}}(1)) \longrightarrow H^{q-1}(\Sigma_k, \A{p+q-1}M_{\sO_{\Sigma_k}(1)} \otimes \sO_{\Sigma_k}(1)) \\
\longrightarrow H^q(\Sigma_{k+1}, \A{p+q-1}M_{\sO_{\Sigma_{k+1}}(1)} \otimes \sI_{\Sigma_k|\Sigma_{k+1}}(1)) \longrightarrow H^q(\Sigma_{k+1}, \A{p+q-1}M_{\sO_{\Sigma_{k+1}}(1)} \otimes \sO_{\Sigma_{k+1}}(1)).
\end{multline}
It suffices to prove that
\begin{subequations}
\begin{align}
& H^{q-1}(\Sigma_{k+1}, \A{p+q-1}M_{\sO_{\Sigma_{k+1}}(1)} \otimes \sO_{\Sigma_{k+1}}(1)) = 0;\label{vanishing2}\\
& H^q(\Sigma_{k+1}, \A{p+q-1}M_{\sO_{\Sigma_{k+1}}(1)} \otimes \sI_{\Sigma_k|\Sigma_{k+1}}(1)) = 0.\label{vanishing3}
\end{align}
\end{subequations}
First, we check (\ref{vanishing3}). By Lemma \ref{duality},
$$
\begin{array}{l}
H^q(\Sigma_{k+1}, \A{p+q-1}M_{\sO_{\Sigma_{k+1}}(1)} \otimes \sI_{\Sigma_k|\Sigma_{k+1}}(1)) \\[5pt]
= H^{q^*+1}(C_{p^*+q^*} \times C_{k+2}, (N_{p^*+q^*, L} \boxtimes S_{k+2, \omega_C})(-D_{p^*+q^*, k+2})),
\end{array}
$$
where $p^*:= e-p \leq \gamma^{q^*}-1$ and $0 \leq q^* := 2k+2-q \leq k+1$. By the Leray spectral sequence for $\pr_1 \colon C_{p^*+q^*} \times C_{k+2} \to C_{p^*+q^*}$, it is enough to confirm that
$$
H^i(C_{p^*+q^*}, M_{p^*+q^*}^{q^*+1-i} S_{k+2, \omega_C} \otimes N_{p^*+q^*, L}) =0~~\text{ for $0 \leq i \leq q^*+1$}.
$$
When $i>0$, this follows from Lemma \ref{Fujita}. For the case $i=0$, we apply Lemma \ref{H^iS_k,N_k} to see that
$$
H^{q^*+1}(C_{k+2}, S_{k+2, \omega_C(-\xi)}) = S^{k+1-q^*} H^0(C, \omega_C(-\xi)) \otimes \A{q^*+1} H^1(C, \omega_C(-\xi))
$$
for any $\xi \in C_{p^*+q^*}$. Proposition \ref{gonseq-h^1} says that $H^1(C, \omega_C(-\xi)) \leq q^*$ since $\gamma^{q^*} \geq p^*+1$, so we obtain
$$
H^{q^*+1}(C_{k+2}, S_{k+2, \omega_C(-\xi)}) = 0.
$$
Thus $M_{p^*+q^*}^{q^*+1} S_{k+2, \omega_C} = 0$, and we obtain (\ref{vanishing3}). To finish the proof of  (\ref{vanishing1}), we proceed by induction on $q-k-1$. If $q=k+1$, then clearly 
$$
H^{k}(\Sigma_{k+1}, \A{p+k}M_{\sO_{\Sigma_{k+1}}(1)} \otimes \sO_{\Sigma_{k+1}}(1))  = K_{p,k+1}(\Sigma_{k+1}, \sO_{\Sigma_{k+1}}(1))=0,
$$
i.e., (\ref{vanishing2}) holds. Thus (\ref{vanishing1}) follows in this case. Suppose that $q \geq k+2$. Lemma \ref{basicgonseq} $(1)$ implies that $e - \gamma^{2k+2-q} +1 \geq (e-2) - \gamma^{2k+4-q} + 1$. By induction hypothesis, 
$$
H^{q-1}(\Sigma_{k+1}, \A{p+q-1}M_{\sO_{\Sigma_{k+1}}(1)} \otimes \sO_{\Sigma_k}(1))  = K_{p,q}(\Sigma_{k+1}, \sO_{\Sigma_{k+1}}(1)) = 0,
$$
i.e., (\ref{vanishing2}) holds. Thus (\ref{vanishing1}) follows.

%%%%%%%%%%%%%%%%%%%%%%%%%%%%%%%%%%%%%%%%%%%%%%%%%%%%%%%%
\subsection*{Nonvanishing}
%%%%%%%%%%%%%%%%%%%%%%%%%%%%%%%%%%%%%%%%%%%%%%%%%%%%%%%%
We show that\\[-20pt]

\begin{small}
\begin{equation}\label{nonvanishing}
K_{p,q}(\Sigma_k, \sO_{\Sigma_k}(1)) = H^{q-1}(\Sigma_k, \A{p+q-1} M_{\sO_{\Sigma_k}(1)} \otimes \sO_{\Sigma_k}(1)) \neq 0~~\text{ for $e-g+1 \leq p \leq e - \gamma^{2k+2-q}$}.
\end{equation}
\end{small}

\noindent We have a commutative diagram with exact rows
$$
\begin{tikzcd}
0 \ar[r] & \sI_{\Sigma_k|\Sigma_{k+1}} \ar[r] \ar[d,equal] & \sI_{\Sigma_{k-1}|\Sigma_{k+1}} \ar[r] \ar[d] & \sI_{\Sigma_{k-1}|\Sigma_k} \ar[r] \ar[d] & 0 \\
0 \ar[r] & \sI_{\Sigma_k|\Sigma_{k+1}} \ar[r] & \sO_{\Sigma_{k+1}} \ar[r] & \sO_{\Sigma_k} \ar[r] & 0.
\end{tikzcd}
$$
This gives a commutative diagram
$$
\begin{tikzcd}
H^{q-1}(\Sigma_k, \A{p+q-1}M_{\sO_{\Sigma_k}(1)}\otimes \sI_{\Sigma_{k-1}|\Sigma_k}(1)) \ar[d] \ar[r,"\varphi"] & H^q(\Sigma_{k+1}, \A{p+q-1}M_{\sO_{\Sigma_{k+1}}(1)}\otimes \sI_{\Sigma_k|\Sigma_{k+1}}(1)) \ar[d,equal]  \\
H^{q-1}(\Sigma_k, \A{p+q-1}M_{\sO_{\Sigma_k}(1)}\otimes \sO_{\Sigma_k}(1)) \ar[r]& H^q(\Sigma_{k+1}, \A{p+q-1}M_{\sO_{\Sigma_{k+1}}(1)} \otimes \sI_{\Sigma_k|\Sigma_{k+1}}(1)).
\end{tikzcd}
$$
It is enough to prove that the map $\varphi$ is nonzero. For this purpose, considering the commutative diagram (\ref{additionmapBertram}), we regard $\varphi$ as a map
\begin{multline*}
\varphi \colon H^{q-1}(B_k, \A{p+q-1}M_{H_{k}}\otimes \sO_{B_k}(H_{k}-Z_{k-1} ))\otimes H^0(C, \sO_C) \\\longrightarrow H^q(B_{k+1}, \A{p+q-1}M_{H_{k+1}}\otimes \sO_{B_{k+1}}(H_{k+1}-Z_k))
\end{multline*}
In view of Lemma \ref{duality}, the map $\varphi$ is dual to the map
\begin{multline*}
\varphi^{\vee} \colon H^{q^*+1}(C_{p^*+q^*} \times C_{k+2}, (N_{p^*+q^*, L} \boxtimes S_{k+2, \omega_C})(-D_{p^*+q^*, k+2})) \\
\longrightarrow H^{q^*}(C_{p^*+q^*} \times C_{k+1}, (N_{p^*+q^*, L} \boxtimes S_{k+1, \omega_C})(-D_{p^*+q^*, k+1})) \otimes H^1(C, \omega_C),
\end{multline*}
where $\gamma^{q^*} \leq p^*:=e-p \leq g-1$ and $0 \leq q^*:=2k+2-q \leq k+1$. Notice that this map is induced from an injective map
$$
(\id_{C_{p^*+q^*}} \times \sigma_{k+2})^* (N_{p^*+q^*, L} \boxtimes S_{k+2, \omega_C})(-D_{p^*+q^*, k+2}) \longhookrightarrow
(N_{p^*+q^*, L} \boxtimes S_{k+1, \omega_C})(-D_{p^*+q^*,k+1}) \boxtimes \omega_C
$$
of line bundles on $C_{p^*+q^*} \times C_{k+1} \times C$. Lemma \ref{Fujita} says that
$$
H^{i}(C_{p^*+q^*}, M_{p^*+q^*}^j(S_{\ell, \omega_C}) \otimes N_{p^*+q^*,L}) = 0~~\text{ for $i>0, ~j \geq 0,~\ell=k+1$ or $k+2$}.
$$
By the Leray spectral sequences for $\pr_1$, we may think that $\varphi^{\vee}$ is a map
\begin{multline*}
\varphi^{\vee} \colon H^0(C_{p^*+q^*}, M_{p^*+q^*}^{q^*+1}S_{k+2, \omega_C} \otimes N_{p^*+q^*, L}) \\
 \longrightarrow H^0(C_{p^*+q^*}, M_{p^*+q^*}^{q^*} S_{k+1, \omega_C} \otimes N_{p^*+q^*, L}) \otimes H^1(C, \omega_C).
\end{multline*}
Notice that this map is induced from a map
$$
\psi \colon M_{p^*+q^*}^{q^*+1}S_{k+2, \omega_C} \longrightarrow M_{p^*+q^*}^{q^*} S_{k+1, \omega_C}  \otimes H^1(C, \omega_C)
$$
of coherent sheaves on $C_{p^*+q^*}$ tensoring by $N_{p^*+q^*, L}$. As $N_{p^*+q^*, L}$ is sufficiently positive, to prove that the map $\varphi^{\vee}$ is nonzero, it suffices to confirm that the map $\psi$ is nonzero. To this end, we apply Proposition \ref{gonseq-h^1} to see that $\omega_C$ fails to be $(q^*-1)$-weakly $p^*$-very ample since $p^* \geq \gamma^{q^*}$. Then Lemma \ref{exactdivisor} gives an effective divisor $\xi \in C_{p^*+q^*}$ on $C$ such that
$$
h^0(C, \omega_C(-\xi)) =g-p^* \geq 1~~\text{ and }~~h^1(C, \omega_C(-\xi)) = q^*+1.
$$
By Lemma \ref{H^iS_k,N_k}, 
$$
H^{i}(C_{\ell}, S_{\ell, \omega_C(-\xi)}) = S^{\ell-i} H^0(C, \omega_C(-\xi)) \otimes \A{i} H^1(C, \omega_C(-\xi)),
$$
so this cohomology vanishes when $i \geq q^*+2$. By semicontinuity, $h^1 (C, \omega_C(-\xi')) \leq q^*+1$ (and hence $H^{q^*+2}(C_{\ell}, S_{\ell, \omega_C(-\xi')}) = 0$) for $\xi'$ in a neighborhood of $\xi$ in $C_{p^*+q^*}$. By the cohomology and base change,
$$
\rho(\xi)^{q^*+1} \colon M_{p^*+q^*}^{q^*+1}S_{k+2, \omega_C} \otimes \kk(\xi) \longrightarrow H^{q^*+1}(C_{k+2}, S_{k+2, \omega_C(-\xi)})
$$
is an isomorphism. We have a commutative diagram
$$
\begin{tikzcd}[column sep=1.8cm] 
M_{p^*+q^*}^{q^*+1}S_{k+2, \omega_C} \otimes \kk(\xi) \ar[r,"\psi \otimes \kk(\xi)"] \ar[d,"\rho(\xi)^{q^*+1}",swap, "\rotatebox{270}{$\simeq$}"'] & M_{p^*+q^*}^{q^*}S_{k+1, \omega_C} \otimes H^1(C, \omega_C) \otimes \kk(\xi)\ \ar[d] \\
H^{q^*+1}(C_{k+2}, S_{k+2, \omega_C(-\xi)}) \ar[r] & H^{q^*}(C_{k+1}, S_{k+1, \omega_C(-\xi)}) \otimes H^1(C, \omega_C).
\end{tikzcd}
$$
We reduce the problem to checking that the bottom map is nonzero. To this end, note that the bottom map can be identified with the map
\begin{multline*}
\id_{S^{k+1-q^*} H^0(C, \omega_C(-\xi))} \otimes \delta \colon S^{k+1-q^*} H^0(C, \omega_C(-\xi)) \otimes \A{q^*+1} H^1(C, \omega_C(-\xi))\\
 \longrightarrow S^{k+1-q^*} H^0(C, \omega_C(-\xi))  \otimes \A{q^*} H^1(C, \omega_C(-\xi)) \otimes H^1(C, \omega_C),
\end{multline*}
where $\delta$ is a Koszul-like map. For a surjective map 
$$
\eta \colon H^1(C, \omega_C(-\xi)) \xrightarrow{~\cdot \xi~} H^1(C, \omega_C),
$$ 
let $s_1, \ldots, s_{q^*+1}$ be a basis of $H^1(C, \omega_C(-\xi))$ with $\eta(s_1) = \cdots = \eta(s_{q^*}) = 0$ but $\eta(s_{q^*+1}) \neq 0$. Then
$$
\delta(s_1 \wedge \cdots \wedge s_{q^*} \wedge s_{q^*+1}) = (-1)^{q^*} s_1 \wedge \cdots \wedge s_{q^*} \otimes \eta(s_{q^*+1}) \neq 0.
$$
Thus the bottom map $\id_{S^{k+1-q^*} H^0(C, \omega_C(-\xi))} \otimes \delta$ in the above commutative diagram is nonzero. Therefore, the map $\varphi^{\vee}$ (and hence $\varphi$) is nonzero, so (\ref{nonvanishing}) follows.

%%%%%%%%%%%%%%%%%%%%%%%%%%%%%%%%%%%%%%%%%%%%%%%%%%%%%%%%
\section{Complements and Questions}\label{secquestions}
%%%%%%%%%%%%%%%%%%%%%%%%%%%%%%%%%%%%%%%%%%%%%%%%%%%%%%%%
In this section, we present some additonal results and problems. We keep using the notations in the previous section. Let $C$ be a smooth projective curve of genus $g \geq 2$, and $L$ be a line bundle on $C$ with $\deg L \geq 2g+2k+1$. We denote by $\Sigma_k$ the $k$-th secant variety of $C$ in $\nP^{2k+1+e} = \nP H^0(C, L)$.
Consider the case that $k=0$. Recall from \cite[Theorem 1.1]{Rathmann} that if $H^1(C, L \otimes \omega_C^{-1})=0$, then
$$
K_{p,1}(C, L) \neq 0~~\Longleftrightarrow~~1\leq p \leq e-\gon(C)+1.
$$
Recall from \cite[Theorem (4.a.1)]{Green}, \cite[Theorem 2]{GL} that if $H^0(C, L \otimes \omega_C^{-1}) \neq 0$, then
$$
K_{p,2}(C, L) \neq 0~~\Longleftrightarrow~~e-g+1 \leq p \leq e.
$$
Thus Theorem \ref{main} holds for $k=0$ as soon as $\deg L \geq 4g-3$.

\begin{problem}
Find an effective bound for $\deg L$ such that the conclusion of Theorem \ref{main} holds. 
\end{problem}

We do not attempt to make a conjecture for what the best bound for $\deg L$ should be, but we expect that it would be linear in $g$. Here we give answers for some partial cases.

%%%%%%%%%%%%%%%%%%%%%%%%%%%%%%%%%%%%%%%%%%%%%%%%%%%%%%%%
\subsection*{Effective Nonvanishing for $q=k+1$}
%%%%%%%%%%%%%%%%%%%%%%%%%%%%%%%%%%%%%%%%%%%%%%%%%%%%%%%%
Recall from Lemma \ref{basicgonseq} $(2)$ that $\gamma^{k+1}(C) = g$ for $k \geq g-1$. If $k \geq g-1$ and $\deg L \geq 2g+2k+1$, then \cite[Theorem 1.2]{ENP1} implies that
$$
K_{p,k+1}(\Sigma_k, \sO_{\Sigma_k}(1)) \neq 0~~\text{ for $1 \leq p \leq e-\gamma^{k+1}(C)$}.
$$
Thus we assume that $k \leq g-2$. On the other hand, Sidman--Vermeire \cite[Theorem 1.2]{SV1} proved that if $L=L_1 \otimes L_2$, where $L_1, L_2$ are line bundles on $C$ with $s+1:=h^0(C, L_1)\geq k+2$ and $t+1 := h^0(C, L_2)\geq k+2$, then 
$$
K_{p,k+1}(\Sigma_k, \sO_{\Sigma_k}(1)) \neq 0~~\text{ for $1 \leq p \leq s+t-2k-1$}.
$$
This yields the following effective nonvanishing statement:

\begin{Prop}\label{effnonvanq=k+1}
Assume that $k \leq g-2$ and $\deg L \geq  2g+\gamma^{k+1}(C)+k$. Then
$$
K_{p,k+1}(\Sigma_k, \sO_{\Sigma_k}(1)) \neq 0~~\text{ for $1 \leq p \leq e-\gamma^{k+1}(C)$}.
$$
\end{Prop}

\begin{proof}
By Lemma \ref{basicgonseq} $(2)$, $\gamma^{k+1}(C) \geq k+1$, so $\deg L \geq 2g+2k+1$. 
We write $\deg L = 2g+\gamma^{k+1}(C)+k+\ell$ for some integer $\ell \geq 0$. Then  $e=g+\gamma^{k+1}(C)+\ell-k-1$. Lemma \ref{exactdivisor} gives a line bundle $L_1$ on $C$ with $\deg L_1 = \gamma^{k+1}(C)+k+1$ and  $s+1:=h^0(C, L_1)=k+2$. Let $L_2:=L \otimes L_1^{-1}$ so that $L=L_1 \otimes L_2$. Then $\deg L_2 = 2g-1+\ell$ and $t+1:=h^0(C, L_2) = g+\ell$. Note that 
$$
s+t-2k-1 = g+\ell-k -1 =e-\gamma^{k+1}(C).  
$$
Thus the proposition follows from  \cite[Theorem 1.2]{SV1}.
\end{proof}

\begin{Rmk}
Assume that $k \leq g-2$. By Lemma \ref{basicgonseq} $(2)$, $\gamma^{k+1}(C) \leq g-1$. Then Proposition \ref{effnonvanq=k+1} holds when $\deg L \geq 4g-3$.
\end{Rmk}

%%%%%%%%%%%%%%%%%%%%%%%%%%%%%%%%%%%%%%%%%%%%%%%%%%%%%%%%
\subsection*{Effective Nonvanishing for $q=2k+2$}
%%%%%%%%%%%%%%%%%%%%%%%%%%%%%%%%%%%%%%%%%%%%%%%%%%%%%%%%
Assume that $\deg L \geq 2g+2k+1$. By duality, we have
$$
K_{p,2k+2}(\Sigma_k, \sO_{\Sigma_k}(1)) = K_{e-p, 0}(\Sigma_k, \omega_{\Sigma_k}; \sO_{\Sigma_k}(1))^{\vee}.
$$
Note that if $K_{g-1, 0}(\Sigma_k, \omega_{\Sigma_k}; \sO_{\Sigma_k}(1)) \neq 0$, then $K_{p,2k+2}(\Sigma_k, \sO_{\Sigma_k}(1)) \neq 0$ for $e-g+1 \leq p \leq e$. 
We need to find an effective bound on $\deg L$ for
$$
K_{g-1,0}(\Sigma_k, \omega_{\Sigma_k}; \sO_{\Sigma_k}(1)) = H^0(\Sigma_k, \A{g-1} M_{\sO_{\Sigma_k}(1)} \otimes \omega_{\Sigma_k})
 \neq 0.
 $$
Notice that $K_{g-1,0}(\Sigma_k, \omega_{\Sigma_k}; \sO_{\Sigma_k}(1))$ is the kernel of the Koszul differential
$$
\delta \colon \A{g-1} H^0(\Sigma_k, \sO_{\Sigma_k}(1)) \otimes H^0(\Sigma_k, \omega_{\Sigma_k}) \longrightarrow \A{g-2} H^0(\Sigma_k, \sO_{\Sigma_k}(1)) \otimes H^0(\Sigma_k, \omega_{\Sigma_k}(1)).
$$
In view of Proposition \ref{Ein}, $\delta$ can be identified with the map
$$
\delta \colon \A{g-1} H^0(C, L) \otimes S^{k+1} H^0(C, \omega_C) \longrightarrow \A{g-2} H^0(C, L) \otimes H^0(C, L \otimes \omega_C) \otimes S^k H^0(C, \omega_C)
$$
given by
$$
\delta(s_1 \wedge \cdots \wedge s_{g-1} \otimes f)=\sum_{i=1}^{g-1}\sum_{j=1}^g (-1)^{i-1}s_1 \wedge \cdots \wedge \widehat{s_i} \wedge \cdots \wedge s_{g-1} \otimes s_i x_j \otimes \frac{\partial f}{\partial x_j},
$$
where $x_1, \ldots, x_g$ is a basis of $H^0(C, \omega_C)$. The following gives an answer to a question of Sidman--Vermeire in \cite[p.164]{SV2}. 

\begin{Prop}\label{effnonvanq=2k+2}
We have the following:
\begin{enumerate}
\item If $k$ is even, then there is an injective map
$$
S^{g-1}H^0(C,L\otimes\omega_C^{-k-1})\longhookrightarrow K_{g-1,0}(\Sigma_k,\omega_{\Sigma_k};\sO_{\Sigma_k}(1)).
$$
\item If $k$ is odd, then there is an injective map
$$
\A{g-1}H^0(C,L\otimes\omega_C^{-k-1}) \longhookrightarrow K_{g-1,0}(\Sigma_k,\omega_{\Sigma_k};\sO_{\Sigma_k}(1)).
$$
\end{enumerate}
In particular, if
$$
h^0(C,L\otimes\omega_C^{-k-1})\geq
\begin{cases}
1&\textup{when $k$ is even}\\
g-1&\textup{when $k$ is odd,}
\end{cases}
$$
then
$$
K_{p,2k+2}(\Sigma_k, \sO_{\Sigma_k}(1)) \neq 0~~\text{ for $e-g+1 \leq p \leq e$}.
$$
\end{Prop}

\begin{proof}
First, we recall some notations from multilinear algebra.
Let $V$ be a vector space over $\kk$, and
$$
T^m V := \underbrace{V \otimes \cdots \otimes V}_{m~\text{times}}~~\text{ for any integer $m \geq 0$}.
$$
Since $\operatorname{char}(\kk)=0$, there are natural splitting injective $\kk$-linear maps
\begin{align*}
& \textup{alt} \colon \A{m} V \longhookrightarrow T^mV,~~v_1 \wedge \cdots \wedge v_m \longmapsto \sum_{\sigma \in \mathfrak{S}_m } \textup{sign}(\sigma) v_{\sigma(1)} \otimes \cdots \otimes v_{\sigma(m)};\\
& \textup{sym} \colon S^m V \longhookrightarrow T^mV, ~~v_1\cdots v_m \longmapsto \sum_{\sigma \in \mathfrak{S}_m} v_{\sigma(1)} \otimes \cdots \otimes v_{\sigma(m)}.
\end{align*}
Put $\textup{Alt}^m V:=\textup{alt} (\A{m} V)$ and $\textup{Sym}^m V:=\textup{sym}(S^m V)$. 

\medskip

Now, write $L_0:=L\otimes\omega_C^{-k-1}$, and let
$$
R:=\bigoplus_{i,j\geq 0}H^0(C,L_0^i\otimes\omega_C^j).
$$
We have the following commutative diagram\\[-20pt]

\begin{small}
$$
\begin{tikzcd}
\A{g-1}H^0(C,L)\otimes S^{k+1}H^0(C,\omega_C) \ar[r,"\delta"] \ar[d, hook, "\textup{alt} \otimes \id_{S^{k+1}H^0(C, \omega_C)}", swap] & \A{g-2}H^0(C,L)\otimes H^0(C,L\otimes\omega_C)\otimes S^kH^0(C,\omega_C) \ar[d, hook, "\textup{alt} \otimes \id_{H^0(C, L \otimes \omega_C)} \otimes \id_{S^{k}H^0(C, \omega_C)}"] \\
T^{g-1}R\otimes S^\ast H^0(C,\omega_C) \ar[r,"d",swap] & T^{g-1}R\otimes S^\ast H^0(C,\omega_C),
\end{tikzcd}
$$
\end{small}

\noindent where the bottom map $d$ is defined by
$$
d(s_1\otimes\cdots\otimes s_{g-2}\otimes s_{g-1}\otimes f)=\sum_{i=1}^g s_1\otimes\cdots\otimes s_{g-2}\otimes s_{g-1}x_i\otimes\frac{\partial f}{\partial x_i}.
$$
Notice that there is a canonical ring structure on $T^{g-1}R\otimes S^\ast H^0(C,\omega_C)$ and the operator $d$ on $T^{g-1}R\otimes S^\ast H^0(C,\omega_C)$ satisfies the chain rule. Consider the alternating tensor
$$
\textup{alt}(x_1\otimes\cdots\otimes x_g)\in\textup{Alt}^gH^0(C,\omega_C)\subseteq T^{g-1}R\otimes S^\ast H^0(C,\omega_C).
$$

\medskip

Suppose that $k$ is even. We may assume that $H^0(C, L \otimes \omega_C^{-k-1}) = H^0(C, L_0) \neq 0$. Let $\alpha_0\in S^{g-1}H^0(C,L_0)$ be any nonzero element, and
$$
\alpha:=(\textup{sym}(\alpha_0)\otimes 1)(\textup{alt}(x_1\otimes\cdots\otimes x_g))^{k+1}\in T^{g-1}R\otimes S^\ast H^0(C,\omega_C).
$$
On the factor $T^{g-1}R$, the element $\textup{sym}(\alpha_0)\otimes 1$ is symmetric, and the element $\textup{alt}(x_1\otimes\cdots\otimes x_g)$ is alternating. Thus $\alpha$ is alternating, that is, $\alpha\in\textup{Alt}^{g-1}H^0(C,L)\otimes S^{k+1}H^0(C,\omega_C)$. On the other hand, by the chain rule, $d\alpha=0$ since $d(\textup{sym}(\alpha_0)\otimes 1)=0$ and $d(\textup{alt}(x_1\otimes\cdots\otimes x_g))=0$. As $T^{g-1}R\otimes S^\ast H^0(C,\omega_C)$ is an integral domain, $\alpha$ is a nonzero element. We have shown that there is an element $\alpha' \in \A{g-1} H^0(C, L) \otimes S^{k+1} H^0(C, \omega_C)$ such that $\delta(\alpha') = 0$. By sending $\alpha_0$ to $\alpha'$, we obtain the injective map in $(1)$. Suppose that $k$ is odd. Replacing $\textup{sym}(\alpha_0)$ with $\textup{alt}(\alpha_0)$ in the definition of $\alpha$, we obtain the injective map in $(2)$. 
\end{proof}

If $C$ is a hyperelliptic curve, then there is a morphism $\tau \colon C \to \nP^1$ of degree two such that $\tau^* \sO_{\nP^1}(g-1)= \omega_C$. Let $P:=\tau^* \sO_{\nP^1}(1)$ so that $\omega_C = P^{g-1}$. In this case, we can improve the previous proposition as follows.

\begin{Prop}\label{effnonvanq=2k+2hyperelliptic}
Assume that $C$ is a hyperelliptic curve. If $H^0 (C, L \otimes P^{-g-k+1} ) \neq 0$, then 
$$
K_{p,0}(\Sigma_k, \omega_{\Sigma_k}; \sO_{\Sigma_k}(1)) \neq 0~~\text{ for $0 \leq p \leq g-1$}.
$$
\end{Prop}

\begin{proof}
Let $A:=\sO_{\nP^1}(g-1)$ and $B:=\sO_{\nP^1}(g+k-1)$. Then $H^0(C, L \otimes \tau^* B^{-1}) \neq 0$. We have a commutative diagram
$$
\begin{tikzcd}
\A{g-1} H^0(\nP^1, B) \otimes S^{k+1}H^0(\nP^1, A) \ar[r, "\delta' "] \ar[hook, d] &
\A{g-2} H^0(\nP^1, B) \otimes H^0(\nP^1, B \otimes A)  \otimes S^{k} H^0(\nP^1, A) \ar[hook, d] \\
\A{g-1} H^0(C, L) \otimes S^{k+1}H^0(C, \omega_C) \ar[r, "\delta", swap]   &
\A{g-2} H^0(C, L) \otimes H^0(C, L \otimes \omega_C)  \otimes S^{k} H^0(C, \omega_C) .
\end{tikzcd}
$$
It suffices to show that the upper horizontal map $\delta'$ is not injective. 
To this end, notice that $\delta'$ can be identified with the map
$$
\delta' \colon \A{g-1} H^0(\nP^1, B) \otimes H^0(\nP^{k+1},  S_{k+1, A})
 \longrightarrow \A{g-2}H^0(\nP^1, B) \otimes H^0(\nP^{k+1},  E_{k+1, B} \otimes S_{k+1, A})
$$
by regarding $\nP^{k+1} = (\nP^1)_{k+1}$.
Thus we obtain
$$
\ker(\delta') = H^0(\nP^{k+1}, \A{g-1} M_{E_{k+1, B}} \otimes S_{k+1, A}).
$$
Since 
$$
M_{E_{k+1, B}} = H^0(\nP^1, \sO_{\nP^1}(g-2)) \otimes \sO_{\nP^{k+1}}(-1)~~\text{ and }~~S_{k+1, A} = \sO_{\nP^{k+1}}(g-1),
$$ 
it follows that $\ker(\delta') = H^0(\nP^{k+1}, \sO_{\nP^{k+1}}) \neq 0$.
\end{proof}

\begin{Rmk}\label{rmk:effnonvanq=2k+2}
The last part of Proposition \ref{effnonvanq=2k+2} holds  as soon as
$$
\deg L \geq \begin{cases}
g+(k+1)(2g-2)&\textup{when $k$ is even}\\
2g-2 + (k+1)(2g-2)&\textup{when $k$ is odd}.
\end{cases}
$$
When $C$ is hyperelliptic, Proposition \ref{effnonvanq=2k+2hyperelliptic} holds as soon as $\deg L \geq 3g+2k-2$.
\end{Rmk}

\begin{Ex}
Suppose that $C$ is nonhyperelliptic and $L:=\omega_C(D)$, where $D$ is a general divisor of degree $g-1$ so that $h^0(C, \sO_C(D))=h^1(C, \sO_C(D))=0$. Then $\deg L = 3g-3$, and $K_{g-1, 0}(C, \omega_C; L) = 0$ by  \cite[Theorem 2]{GL}.
For an integer $1 \leq k \leq (g-4)/2$, we have $\deg L \geq 2g+2k+1$. Consider the commutative diagram\\[-20pt]

\begin{footnotesize}
$$
\begin{tikzcd}[column sep=2.3cm] 
\A{g-1} H^0(C, L) \otimes S^{k+1} H^0(C, \omega_C) \ar[r,"\delta"] \ar[d,"\id_{\wedge^{g-1}H^0(C, L)} \otimes m",swap] & \A{g-2} H^0(C, L) \otimes H^0(C, L \otimes \omega_C) \otimes S^k H^0(C, \omega_C) \ar[d,equal] \\
\A{g-1} H^0(C, L) \otimes H^0(C, \omega_C) \otimes S^k H^0(C, \omega_C) \ar[r,"\delta \otimes \id_{S^k H^0(C, \omega_C)}",swap] & \A{g-2} H^0(C, L) \otimes H^0(C, L \otimes \omega_C) \otimes S^k H^0(C, \omega_C),
\end{tikzcd}
$$
\end{footnotesize}

\noindent where  $m \colon S^{k+1} H^0(C, \omega_C) \to H^0(C, \omega_C) \otimes S^k H^0(C, \omega_C)$ is given by 
$m(f) = \sum_{j=1}^g x_j \otimes \partial f/\partial x_j$. As $K_{g-1, 0}(C, \omega_C; L)$ is the kernel of the Koszul differential
$$
\delta \colon \A{g-1} H^0(C, L) \otimes H^0(C, \omega_C) \longrightarrow \A{g-2} H^0(C, L) \otimes H^0(C, L \otimes \omega_C),
$$
we see that $K_{g-1, 0}(\Sigma_k, \omega_{\Sigma_k}; \sO_{\Sigma_k}(1)) \subseteq K_{g-1, 0}(C, \omega_C; L) \otimes S^k H^0(C, \omega_C)$. Thus we obtain $K_{g-1, 0}(\Sigma_k, \omega_{\Sigma_k}; \sO_{\Sigma_k}(1))=0$ in this case.
\end{Ex}

%%%%%%%%%%%%%%%%%%%%%%%%%%%%%%%%%%%%%%%%%%%%%%%%%%%%%%%%
\subsection*{Effective Vanishing for $q=2k+1$}
%%%%%%%%%%%%%%%%%%%%%%%%%%%%%%%%%%%%%%%%%%%%%%%%%%%%%%%%
Let $c:=\gamma^1(C)=\gon(C)-1$. Then $\omega_C$ is $(c-1)$-very ample. For any $1 \leq p \leq c$, as $h^0(C, \omega_C(-\xi)) = g-p$ for all $\xi \in C_p$, we see that $M_{E_{p, \omega_C}} = \pr_{1,*} (\sO_{C_p} \boxtimes \omega_C)(-D_{p,1})$ is a vector bundle on $C_p$. First, we prove the following vanishing result:

\begin{Prop}[{cf. \cite[Proposition 2.1]{EL3}}]\label{cohvansymmM}
Assume that $\deg L \geq (c^2+kc+k+1)(g-1)+1$. Then
$$
H^i(C_p, S^k M_{E_{p, \omega_C}} \otimes N_{p,L}) = 0~~\text{ for $i > 0$ and $1 \leq p \leq c$}.
$$
\end{Prop}

\begin{proof}
Let $V \subseteq H^0(C, \omega_C)$ be a general subspace of dimension $2p$ so that the evaluation map $\operatorname{ev} \colon V \otimes C_p \to E_{p, \omega_C}$ is surjective, and $M_V$ be the kernel of the evaluation map. Then $M_V$ is a vector bundle of rank $p$ on $C_p$. We have a short exact sequence
$$
\begin{tikzcd}
0 \ar[r]  & M_V  \ar[r] & M_{E_{p, \omega_C}} \ar[r] & (H^0(C, \omega_C)/V) \otimes \sO_{C_p} \ar[r] & 0.
\end{tikzcd}
$$
By considering the filtration of $S^k M_{E_{p, \omega_C}}$ associated to this short exact sequence, we reduce the problem to proving that
\begin{equation}\label{eq:vanS^jM_V}
H^i(C_p, S^j M_V \otimes N_{p,L}) = 0~~\text{ for $i>0$ and $0 \leq j \leq k$}.
\end{equation}
Notice that $M_V \otimes N_{p, \omega_C}$ is globally generated and $A_j:=N_{p,L} \otimes N_{p, \omega_C}^{-(p+j)}$ is ample for $0 \leq j \leq k$ (see \cite[Proof of Proposition 2.1]{EL3}). Then
$$
S^j M_V \otimes N_{p,L} = N_{p, \omega_C} \otimes S^j (M_V \otimes N_{p, \omega_C}) \otimes \det (M_V \otimes N_{p, \omega_C}) \otimes A_j.
$$
As $\omega_{C_p} = N_{p, \omega_C}$, the required cohomology vanishing (\ref{eq:vanS^jM_V}) follows from Griffiths vanishing \cite[Variant 7.3.2]{positivity}.
\end{proof}

\begin{Prop}\label{effvanq=2k+1}
Assume that $\deg L \geq \big( c^2 + (c+1)(k+1+\lfloor c/2 \rfloor)+1 \big)(g-1)+1$. Then
$$
K_{p,2k+1}(\Sigma_k, \sO_{\Sigma_k}(1)) =0~~\text{ for $p \geq e-c+1$}.
$$
\end{Prop}

\begin{proof}
Arguing as in the proof of Theorem \ref{main}, we reduce the problem to proving that
$$
H^{2k+1}(\Sigma_{k+i}, \A{p+2k} M_{\sO_{\Sigma_{k+i}}(1)} \otimes \sI_{\Sigma_{k+i-1}|\Sigma_{k+i}}(1)) = 0~~\text{ for $i \geq 1$},
$$
which is equivalent to
$$
H^{2i}(C_{p^*} \times C_{k+1+i}, (N_{p^*, L} \boxtimes S_{k+1+i, \omega_C})(-D_{p^*, k+1+i})) = 0~~\text{ for $i \geq 1$},
$$
where $p^*:=e-p+1 \leq c$, by Lemma \ref{duality}. Using Lemma \ref{H^iS_k,N_k}, a similar argument of the proof of Lemma \ref{wedgeM_{E_{k,L}}} yields that
$$
M_{p^*}^j  S_{k+1+i, \omega_C} = \begin{cases} S^{k+1+i} M_{E_{p^*, \omega_C}} & \text{if $j=0$} \\ S^{k+i} M_{E_{p^*, \omega_C}} & \text{if $j=1$} \\ 0 & \text{if $j \geq 2$}.\end{cases}
$$
Then it is enough to show that
$$
H^{2i-1+j}(C_{p^*}, S^{k+i+j} M_{E_{p^*, \omega_C}} \otimes N_{p^*, L}) = 0~~\text{ for $i \geq 1$ and $j=0,1$},
$$
but this follows from Proposition \ref{cohvansymmM}.
\end{proof}

\begin{Rmk}
In view of \cite[Theorem 4.1]{ENP1} and \cite[Theorem 3.1]{Rathmann}, we expect that Propositions \ref{cohvansymmM} and \ref{effvanq=2k+1} hold under a much weaker assumption. 
\end{Rmk}

\begin{Ex}
Let $C$ be a smooth plane quartic curve. Then the genus $g$ of $C$ is $3$, and $\gamma^0(C)=0, \gamma^1(C)=2, \gamma^2(C)=2$. Let $L_1:=\omega_C^3,L_2:=L_1(-x),L_3:=L_1(-x-y)$, where $x,y$ are random points on $C$. Note that $\deg L_1 = 12,~\deg L_2 = 11,~\deg L_3 = 10$.  A \texttt{Macaulay2} \cite{GS} computation shows that the Betti tables of $R(\Sigma_1, \sO_{\Sigma_1}(1))$ for $L=L_1, L_2, L_3$ are the following:

\smallskip

\begin{footnotesize}
\begin{table}[h]
\texttt{\begin{tabular}{c|ccccccc}
       & $0$ & $1$ & $2$ & $3$ & $4$ & $5$ & $6$ \\ \hline
$0$ & 1 & - & - & - & - & - & - \\ 
$1$ & - & - & - & - & - & - & - \\ 
$2$ & - & 38 & 108 & 102 & 10 & - & - \\ 
$3$ & - & - & - & - & 30 & - & - \\ 
$4$ & - & - & - & - & 3 & 18 & 6 \\ 
\end{tabular}}\quad~\quad~\quad~\quad
\texttt{\begin{tabular}{c|cccccc}
       & $0$ & $1$ & $2$ & $3$ & $4$ & $5$ \\ \hline
$0$ & 1 & - & - & - & - & -  \\ 
$1$ & - & - & - & - & - & -  \\ 
$2$ & - & 20 & 36 & 6 & - & -  \\ 
$3$ & - & - & - & 20 & 1 & -  \\ 
$4$ & - & - & - & 1 & 15 & 6  \\ 
\end{tabular}}\quad~\quad~\quad~\quad
\texttt{\begin{tabular}{c|ccccc}
       & $0$ & $1$ & $2$ & $3$ & $4$ \\ \hline
$0$ & 1 & - & - & - & - \\ 
$1$ & - & - & - & - & - \\ 
$2$ & - & 8 & 3 & - & - \\ 
$3$ & - & - & 12 & 2 & - \\ 
$4$ & - & - & - & 12 & 6 \\ 
\end{tabular}}\\[10pt]
\caption{\label{table3}The Betti tables of $R(\Sigma_1, \sO_{\Sigma_1}(1))$}
\end{table}
\end{footnotesize}

\noindent \\[-25pt] When $L=L_3$, we see that $K_{g-1, 0}(\Sigma_1, \omega_{\Sigma_1}; \sO_{\Sigma_1}(1)) =0$. In this case, $h^0(C, L \otimes \omega_C^{-2})=h^0(C, \omega_C(-x-y))=1 < 2=g-1$. This shows that the condition in Proposition \ref{effnonvanq=2k+2} is sharp. On the other hand, notice that $K_{1,1}(\Sigma_1, \omega_{\Sigma_1}; \sO_{\Sigma_1}(1)) \neq 0$ for $L=L_2, L_3$; in other words, the conclusion of Proposition \ref{effvanq=2k+1} does not hold. However, $K_{1,1}(C, \omega_C; L)=0$ for $L=L_2, L_3$ by \cite[Theorem 1.1]{Rathmann} since $\deg L \geq 9 = 4g-3$.
\end{Ex}

We now turn to the quantitative study of the nonzero Betti numbers 
$$
\kappa_{p,q}(\Sigma_k, \sO_{\Sigma_k}(1)):=\dim K_{p,q}(\Sigma_k, \sO_{\Sigma_k}(1)).
$$ 
It would be exceedingly interesting to know whether there is a uniform asymptotic behavior of $\kappa_{p,q}(\Sigma_k, \sO_{\Sigma_k}(1))$ as the positivity of $L$ grows. If so, one may further ask what kind of geometry of $C$ is related to this asymptotic behavior.

\medskip

For integers $m, \ell \geq 1$, we define
$$
\mathcal{L}_m^{\ell} = \mathcal{L}_m^{\ell}(C) :=\{\xi \in C_m \mid h^1(C, \omega_C(-\xi)) \geq \ell\}.
$$
Let $e:=\codim \Sigma_k$. 

\begin{Prop}
Fix an integer $k+1 \leq q \leq 2k+1$. Assume that $L=L_d:=\sO_C(dA +P)$ for an integer $d \gg 0$, where $A$ is an ample divisor on $C$ and $P$ is any divisor on $C$. Then $\kappa_{e-\gamma^{2k+2-q}(C),q}(\Sigma_k, \sO_{\Sigma_k}(1))$ is a polynomial in $d$ of degree $\dim \mathcal{L}_{2k+2-q + \gamma^{2k+2-q}(C)}^{2k+3-q}(C)$. 
\end{Prop}

\begin{proof}
Put $p:=e-\gamma^{2k+2-q}(C)$. By Theorem \ref{main}, 
\begin{align*}
&H^{q-1}(\Sigma_{k+1}, \A{p+q-1}M_{\sO_{\Sigma_{k+1}}(1)} \otimes \sO_{\Sigma_{k+1}}(1)) = K_{p,q}(\Sigma_{k+1}, \sO_{\Sigma_{k+1}}(1)) = 0;\\
&H^{q}(\Sigma_{k+1}, \A{p+q-1}M_{\sO_{\Sigma_{k+1}}(1)} \otimes \sO_{\Sigma_{k+1}}(1)) = K_{p-1,q+1}(\Sigma_{k+1}, \sO_{\Sigma_{k+1}}(1)) = 0.
\end{align*}
Then the exact sequence (\ref{exseqforvan}) shows that\\[-20pt]

\begin{small}
$$
\kappa_{p,q}(\Sigma_{k}, \sO_{\Sigma_{k}}(1)) = h^{q-1}(\Sigma_{k}, \A{p+q-1}M_{\sO_{\Sigma_{k}}(1)} \otimes \sO_{\Sigma_{k}}(1)) = h^{q}(\Sigma_{k+1}, \A{p+q-1} M_{\sO_{\Sigma_{k+1}}(1)} \otimes \sI_{\Sigma_k|\Sigma_{k+1}}(1)).
$$
\end{small}

\noindent By Lemmas \ref{Fujita} and \ref{duality} and the Leray spectral sequence for $\pr_1 \colon C_{p^*+q^*} \times C_{k+2} \to C_{p^*+q^*}$, we have
$$
h^{q}(\Sigma_{k+1}, \A{p+q-1} M_{\sO_{\Sigma_{k+1}}(1)} \otimes \sI_{\Sigma_k|\Sigma_{k+1}}(1)) = h^0(C_{p^*+q^*}, M_{p^*+q^*}^{q^*+1} S_{k+2, \omega_C} \otimes N_{p^*+q^*, L}),
$$
where $p^*:=e-p$ and $q^*:=2k+2-q$. Note that $\Supp M_{p^*+q^*}^{q^*+1} S_{k+2, \omega_C} =  \mathcal{L}_{p^*+q^*}^{q^*+1}(C)$. As we may write $N_{p^*+q^*, L} = N_{p^*+q^*, \sO_C(P)} \otimes S_{p^*+q^*, \sO_C(A)}^d$ and $S_{p^*+q^*, \sO_C(A)}$ is ample, we see that 
$$
h^0(C_{p^*+q^*}, M_{p^*+q^*}^{q^*+1} S_{k+2, \omega_C} \otimes N_{p^*+q^*, L}) = \chi( M_{p^*+q^*}^{q^*+1} S_{k+2, \omega_C} \otimes N_{p^*+q^*, \sO_C(P)} \otimes S_{p^*+q^*, \sO_C(A)}^{d})
$$
is a polynomial in $d$ of degree $\dim  \mathcal{L}_{p^*+q^*}^{q^*+1}(C)$. 
\end{proof}

In the situation of the above proposition, for $e-g+1 \leq p \leq e$, Ein--Lazarsfeld \cite[Theorem C]{EL3} proved that $\kappa_{p,1}(C, \omega_C; L)$ is a polynomial in $d$ (see \cite{Yang} for a higher dimensional generalization). Thus it is natural to ask the following.

\begin{Q}
For $e-g+1 \leq p \leq e$ and $k+1 \leq q \leq 2k+2$, is $\kappa_{p,q}(\Sigma_k, \sO_{\Sigma_k}(1))$ a polynomial in $d:=\deg L$ when $d \gg 0$?
\end{Q}

In some cases, one can compute $\kappa_{p,q}(\Sigma_k, \sO_{\Sigma_k}(1))$ exactly. For instance, $\kappa_{e, 2k+2}(\Sigma_k, \sO_{\Sigma_k}(1)) = {g+k \choose k+1}$ (see \cite[Theorem 1.2]{ENP1}). In the curve case, Kemeny \cite[Theorem 1.1]{Kemeny} proved that if $C$ is a general curve of genus $g \geq 2k-1$ and gonality $k=\gamma^1(C)+1 \geq 4$ and $L$ is a line bundle on $C$ with $\deg L \geq 2g+k$, then 
$$
\kappa_{e-\gamma^1(C), 1}(C, L) = e-\gamma^1(C),
$$
where $e:=h^1(C, L) - 2$ is the codimension of $C$ in $\nP H^0(C, L) = \nP^r$. This theorem can be geometrically interpreted as follows. Let $\tau \colon C \to \nP^1$ be a branched covering of degree $k$. Then the linear spans of the fibers of $\tau$ in $\nP^r$ sweep out a $k$-dimensional scroll $S$ containing $C$. There is a natural injective map $\iota_p \colon K_{p,1}(S, \sO_S(1)) \to K_{p,1}(C, L)$. Kemeny's theorem says that $\iota_{e-\gamma^1(C)}$ is in fact an isomorphism. Along this line, one may ask the following:

\begin{Q}
Fix an integer $k+1 \leq q \leq 2k+1$. Under what conditions, can one compute $\kappa_{e-\gamma^{2k+2-q}(C),q}(\Sigma_k, \sO_{\Sigma_k}(1))$ exactly? In this case, can one find some interesting geometric meaning of spanning Koszul classes of $K_{e-\gamma^{2k+2-q}(C),q}(\Sigma_k, \sO_{\Sigma_k}(1))$?
\end{Q}

For an integer $k\geq 0$, suppose that $C$ is a general curve carrying a unique $(k+1)$-dimensional linear system $|L_1|$ of degree $\gamma^{k+1}(C)+k+1$. Then we expect that 
$$
\kappa_{e-\gamma^{k+1}(C),k+1}(\Sigma_k, \sO_{\Sigma_k}(1))=\binom{e-\gamma^{k+1}(C)+k}{k+1}.
$$
Suppose that the expectation is true. Let $M$ be a matrix given by the multiplication map
$$
H^0(C,L_1)\otimes H^0(C,L\otimes L_1^{-1}) \longrightarrow H^0(C,L),
$$
and $X\subseteq\mathbb P^r$ be the projective variety cut out by $(k+2)$-minors of $M$. Then the natural map 
$$
K_{e-\gamma^{k+1}(C),k+1}(X, \sO_X(1)) \longrightarrow K_{e-\gamma^{k+1}(C),k+1}(\Sigma_k, \sO_{\Sigma_k}(1))
$$
is an isomorphism. We remark that $R(X,\sO_X(1))$ is minimally resolved by the Eagon--Northcott complex associated to $M$. Thus $K_{e-\gamma^{k+1}(C),k+1}(\Sigma_k, \sO_{\Sigma_k}(1))$ is spanned by Koszul classes of the smallest rank $e-\gamma^{k+1}(C)+k+1$ (see \cite[Corollary 4.3]{CK}).

%%%%%%%%%%%%%%%%%%%%%%%%%%%%%%%%%%%%%%%%%%%%%%%%%%%%%%%%
\bibliographystyle{amsalpha}
%%%%%%%%%%%%%%%%%%%%%%%%%%%%%%%%%%%%%%%%%%%%%%%%%%%%%%%%

\end{document}